\documentclass[12pt]{article}

\usepackage{amsmath,amssymb,latexsym, amsfonts, amscd, amsthm}
\usepackage[french]{babel}
\usepackage[latin1]{inputenc}
\theoremstyle{plain}
\newtheorem{mconjecture}{Conjecture}
\newtheorem{mquestion}{Question}
\newtheorem{theorem}{Théorème}[section]
\newtheorem{conjecture}[theorem]{Conjecture}
\newtheorem{proposition}[theorem]{Proposition}
\newtheorem{corollary}[theorem]{Corollaire}

\newtheorem{lemma}[theorem]{Lemme}

\theoremstyle{definition}

\newtheorem{remark}[theorem]{Remarque}

\def\lra{{\longrightarrow}}
\def\SL{{\bf SL}}
\def\GL{{\bf GL}}

\def\PP{{\mathbf P}}

\def\A{\mathbf{A}}

\def\cX{\mathcal{X}}
\def\cG{\mathcal{G}}
\def\cC{\mathcal{C}}
\def\simarrow{{\stackrel{\sim}{\rightarrow}}}

\DeclareMathOperator{\Pic}{Pic} 
\DeclareMathOperator{\rec}{rec} 
\DeclareMathOperator{\Hom}{Hom} 

\DeclareMathOperator{\Eis}{\mathrm{Eis}}
\DeclareMathOperator{\sign}{signe}

\DeclareMathOperator{\Id}{Id}

\newcommand{\Ima}{\mathrm{\,Im\,}}
\newcommand{\Reel}{\mathrm{\,Re\,}}

\newcommand{\mat}[4]{\left( \begin{array}{cc} {#1} & {#2} \\ {#3} & {#4}
\end{array} \right)}

\newcommand{\prim}{\begin{array}{l}\!\!\!\!\! \prime\\ \ \end{array}}

\def\TT{\mathbf{T}}
\def\Z{\mathbf{Z}}

\def\Q{\mathbf{Q}}
\def\C{\mathbf{C}}

\def\R{\mathbf{R}}

\def\bdf{\begin{defn}}
\def\edf{\end{defn}}

\def\cH{\mathcal{H}}

\def\cO{\mathcal{O}}

\def\sk{\vskip 0.2cm}
\def\Gal{{\rm Gal}}
\def\gal{{\rm Gal}}

\def\N{\text{N}}

\def\mb{\mathbb}
\def\mc{\mathcal}

\begin{document}

$$ $$

\begin{center}{ \LARGE Arguments  des unités de Stark et \\
p\'eriodes de s\'eries d'Eisenstein\\
\vspace{0.7cm}
 \large
Pierre Charollois\footnote{ {\tt{charollois@math.jussieu.fr}}\\
Le premier auteur a
 b\'en\'efici\'e du soutien mat\'eriel
 de l'Institut de Math\'ematiques de Bordeaux et  du CRM-ISM
 de Montr\'eal au cours de l'\'elaboration
de cet article.}
 \\ \vspace{0.05cm} Henri Darmon\footnote{
{\tt darmon@math.mcgill.ca} \\
Le travail du second auteur a \'et\'e financ\'e
en partie par le CRSNG et la Chaire James McGill.}\\
\vspace{0.4cm}
25 octobre  2007
}
 \end{center}

\normalsize
 \bigskip

 \tableofcontents

\bigskip

\section*{Introduction}
Soit $K$ un corps de nombres,
et soit
$$S_\infty =\{v_1,\ldots, v_t \} $$
l'ensemble  de ses places archim\'ediennes.
Pour simplifier les \'enonc\'es qui suivent, on supposera (dans l'introduction
seulement)  que le
corps $K$ a pour nombre de classes $1$ au sens restreint.
On choisit,
pour chaque  place de $S_\infty$, 
 un plongement r\'eel ou complexe
associ\'e,
que l'on d\'esignera par le m\^eme symbole,
et l'on pose
(pour $x$ appartenant \`a $K^\times$, et $v\in S_\infty$)
$$ s_v(x):= \left\{
\begin{array}{ll}
\sign(v(x))\in\{-1,1\} & \mbox{ si $v$ est r\'eelle};\\
 1 & \mbox{ si $v$ est complexe}.
\end{array} \right.
$$

Soit $I$ un id\'eal de l'anneau des entiers $\cO_K$ de $K$, et soit
$\cO_{K,+}^\times(I)$ le sous-groupe
du groupe $\cO_{K,+}^\times$ des unit\'es totalement positives de $\cO_K$
form\'e des \'el\'ements qui sont
congrus \`a $1$ modulo $I$.
Pour tout $a\in (\cO_K/I)$, on associe au choix
de signes $s_{v_2},\ldots, s_{v_n}$
la fonction $L$ partielle de Hurwitz:
\begin{equation}
\label{HeckeShimizudef}
L(a, I,s)  :=   (\N I)^s \!\!\!\!
\sum_{\stackrel{x\in \cO_K/\cO_{K,+}^\times(I)}{x\equiv a \  (I)}}
\!\!\!\!\!\!\!\prim\!\!\!\!\!\!\! s_{v_2}(x) \cdots s_{v_t}(x) | \N_{K/\mb Q}(x)|^{-s},
\end{equation}
o\`u le symbole $\sum \  \, \prim\!\!\!\!\!\!\!$ indique que la somme est \`a prendre sur les
\'el\'ements {\em non-nuls}.
Ces fonctions $L$  jouissent  des propri\'et\'es
suivantes:
\begin{enumerate}
\item La fonction $L(a,I,s)$ ne d\'epend que de l'image de
$a$ dans le quotient
$(\cO_K/I)/\cO_{K,+}^\times$, sur lequel op\`ere le groupe
\begin{equation}
\label{eqn:defGIintro}
\cG_I:=(\cO_K/I)^\times/\cO_{K,+}^\times.
\end{equation}
Plus g\'en\'eralement,
pour toute
 unit\'e $\epsilon\in \cO_K^\times$, on a
$$ L(\epsilon a, I,s) = s_{v_2}(\epsilon)\cdots s_{v_t}(\epsilon) L(a,I,s).$$
\item La s\'erie qui d\'efinit $L(a,I,s)$ converge
absolument sur le demi-plan ${\mbox{Re}}(s)>1$, et admet un
prolongement méromorphe \`a tout le plan complexe avec au plus un
p\^ole simple en $s=1$. Si $t\ge 2$, cette fonction est m\^eme holomorphe,
et s'annule en $s=0$.
\end{enumerate}
Pour toute place r\'eelle $v\in S_\infty$, il existe alors une unit\'e
$\epsilon_v\in\cO_{K}^\times$
telle que
 $$s_v(\epsilon_v)=-1, \qquad  s_{v'}(\epsilon_v)=1 \ \mbox{ pour tout } v'\ne v.$$
De plus, la th\'eorie du corps de classes identifie
le groupe $\cG_I$  avec le groupe de Galois d'une extension
ab\'elienne $H$ de $K$, appel\'ee le {\em corps de classes de rayon au sens restreint}
associ\'e \`a $I$.
Soit
$$ \rec: \cG_I \lra \gal(H/K)$$
l'isomorphisme de r\'eciprocit\'e  de la th\'eorie du corps de classes,
et pour tout $v\in S_\infty$, soit
$c_v$ la conjugaison
complexe (``\'el\'ement de Frobenius") associ\'ee \`a la place
$v$,  de sorte que
$$ c_v = \left\{
\begin{array}{ll}

\rec(\epsilon_v) & \mbox{ si $v$ est r\'eelle},
\\
1 & \mbox{ si $v$ est complexe}.
\end{array} \right.
$$
On note $e$ l'ordre du groupe des racines de l'unit\'e dans $H$ et l'on  choisit une place ${\tilde v_1}$
de $H$ au-dessus de la place $v_1$.

La conjecture de Stark \cite{St}, \cite{Ta},
 concerne les
d\'eriv\'ees premi\`eres de $L(a,I,s)$ en $s=0$, et peut s'\'enoncer comme 
suit:
\begin{mconjecture}[Stark]
\label{conj:stark}
Pour tout $a\in (\cO_K/I)$, il
existe une unit\'e $u_a\in \cO_{H}^\times$,
appel\'ee {\em unit\'e de Stark} associ\'ee au couple  $(a,I)$,
telle que
\begin{enumerate}
\item $L'(a,I,0) =\frac 1e  \log |\tilde v_1(u_{a})|^2$;
\item $c_{v_1}(u_a) = u_a$;
\item Si $t\ge 3$,  alors  $c_{v_2} u_a = \cdots = c_{v_t} u_a = u_a^{-1}$;
\item Pour tout $b\in\cG_I$, on a $u_{ab} = \rec(b)^{-1} u_a$.
\end{enumerate}
\end{mconjecture}
On notera que les unit\'es conjecturales $u_a$ d\'ependent du choix de la place
${\tilde v_1}$ au-dessus de $v_1$, mais seulement \`a conjugaison pr\`es par
$\gal(H/K)$.

Si $S_\infty-\{v_1\}$ poss\`ede
une place complexe,   la Conjecture \ref{conj:stark} est trivialement v\'erifi\'ee.
En effet, quand
 $t=2$, la quantit\'e
 $L'(a,I,0)$ ne d\'epend
que  de $I$ et non de $a$, et
s'\'ecrit comme un multiple rationnel du logarithme
d'une unit\'e fondamentale de $K$.
Quand $t> 2$, on a de plus  $L'(a,I,0)=0$, de sorte que
la Conjecture \ref{conj:stark} est v\'erifi\'ee avec
$u_a=1$.

Par cons\'equent, la Conjecture \ref{conj:stark} n'a de l'int\'er\^et
que lorsque  les places $v_2,\ldots, v_t$ sont toutes
r\'eelles, ce qui nous am\`ene \`a distinguer deux cas.

\sk\noindent {\bf 1. Le cas totalement r\'eel}. Si la place $v_1$ est \'egalement  r\'eelle,
 le corps $K$ est  {\em totalement r\'eel.}
A cause de la propri\'et\'e $2$ dans la Conjecture
\ref{conj:stark}, l'expression ${\tilde v_1}(u_a)$
 appartient alors \`a $\R$.
La  Conjecture
\ref{conj:stark} permet donc  d'\'evaluer ce nombre, du moins
au signe
pr\`es. On obtient ainsi, par \'evaluation des
d\'eriv\'ees en $s=0$ des s\'eries $L(a,I,s)$, la construction
analytique d'unit\'es explicites de $H$.
Les conjectures de
 Stark, une fois d\'emontr\'ees,
fourniraient ainsi un \'el\'ement de
``th\'eorie explicite du corps de classes"
  pour les corps de
nombres totalement r\'eels.

\sk\noindent {\bf 2. Le cas ATR.} Si $v_1$ est une place complexe,
on dit que $K$ est un  corps de nombres ATR (``Almost Totally
Real") suivant la terminologie de \cite{darmon-logan}.
 Puisque $c_{v_1}=1$, l'expression
${\tilde v_1}(u_a)$ n'a plus de raison d'\^etre
 r\'eelle {\em a priori},
 et la    Conjecture
\ref{conj:stark} ne permet  d'en \'evaluer que la
{\em valeur absolue}. L'ambigu\"{\i}t\'e de signe du
cas totalement r\'eel s'av\`ere donc
 plus s\'erieuse dans le contexte
ATR, puisqu'elle porte sur {\em l'argument} de
${\tilde v_1}(u_a)$, un \'el\'ement de $\R/(2\pi \Z)$. On est
amen\'e \`a poser la question suivante qui peut servir  de motivation
pour cet article.
\begin{mquestion}
\label{question:stark}
Existe-t-il une formule analytique
explicite pour l'expression
 ${\tilde v_1}(u_a)$ qui appara\^{\i}t dans la
 Conjecture \ref{conj:stark}, lorsque $K$ est ATR?
\end{mquestion}
Une r\'eponse affirmative \`a cette question fournirait une
solution  du $12$-\`eme probl\`eme de Hilbert pour les
extensions ATR.

Dans le premier cas intéressant o\`u $K$ est un corps cubique
complexe, cette question a \'et\'e consid\'er\'ee dans
\cite{dasgupta_senior} et dans \cite{dtv} o\`u la conjecture de
Stark est \'etudi\'ee num\'eriquement, et un progr\`es d\'ecisif a
\'et\'e accompli dans \cite{ren-sczech}.

Le pr\'esent article se penche sur  la Question
\ref{question:stark} lorsque
le corps $K$ est une extension quadratique d'un corps totalement
r\'eel $F$, et lorsque   le corps de rayon $H$  est remplac\'e par
un certain sous-corps---le  corps de classes d'anneau  (``ring class field")
associ\'e \`a $I$ et $K/F$--- dont la d\'efinition sera rappel\'ee
dans la Section \ref{sec:preliminaires}.

Pour motiver notre approche, examinons
ce qui se passe dans le cas le plus simple o\`u $F=\Q$,   o\`u $K$
est un corps quadratique imaginaire de nombre de classes
$1$, et o\`u $I=(m)$ est un id\'eal rationnel engendr\'e
par $m\in \Z$. Au lieu de  porter sur les s\'eries
partielles de Hurwitz,
 les conjectures de cet article vont plut\^ot porter sur les sommes
\begin{equation}
\label{eqn:Llattice}
 \sum_{r\in(\Z/m\Z)}
L(ar,I,s) =: L(M,s) = \ \
 (\N M)^s \!\!\!
\sum_{x\in M/\cO_{K,+}^\times(I)}
\!\!\!\!\!\!\!\prim\!\!\!\!\!| \N_{K/\mb Q}(x)|^{-s},
\end{equation}
o\`u
$$M = \{ x\in \cO_K, \mbox{ tel qu'il existe } r\in \Z \mbox{ avec }
x\equiv ar \pmod{I} \}.   $$
Le module $M$ est un r\'eseau dans $K\subset \C$,
et,
quand $a$ appartient \`a $\cG_I$, il
 forme m\^eme un
module projectif sur l'ordre $\cO_I:= \Z + I\cO_K$.
La s\'erie $L(M,s)$ ne d\'epend que de la classe d'homoth\'etie de ce
r\'eseau; elle est donc d\'etermin\'ee par
l'invariant $\tau =\frac{\omega_2}{\omega_1}$, o\`u $(\omega_1,\omega_2)$
est une $\Z$-base de $M$ choisie de telle mani\`ere que $\tau$
appartienne au demi-plan  de Poincar\'e $\cH$.
La {\em formule limite de Kronecker} fournit les premiers
termes du
d\'eveloppement de $L(M,s) $ en $s=0$ en la reliant au logarithme de la fonction $\eta$ de Dedekind :
\begin{equation}
\label{eq:LimKro}
L(M,s) = -\frac 12 - \frac 12 \left(c_I + \log \textrm{Im}\,(\tau)+
4 \log |\eta(\tau)|\right) s + O(s^2),\end{equation}
o\`u $c_I$ est une constante qui ne d\'epend que de $I$ et pas de
$M.$ \`A des facteurs parasites pr\`es, les d\'eriv\'ees
premi\`eres $L'(M,0)$ sont donc fournies par l'expression
$\log|\eta(\tau)|$. Or, comme $\tau$ appartient \`a $\cH\cap K$,
les produits d'expressions de la forme $\eta(\tau)$ donnent lieu
aux unit\'es elliptiques, que l'on sait \^etre des unit\'es dans
des extensions ab\'eliennes du corps $K$ gr\^ace \`a la th\'eorie
de la multiplication complexe. Plus pr\'ecis\'ement, pour
$\tau_1,\tau_2\in \cH\cap K$, les
expressions de la forme
\begin{equation}
\label{eqn:multeu}
 u(\tau_1,\tau_2):= \eta(\tau_2)/\eta(\tau_1)
\end{equation}
sont des nombres alg\'ebriques, et leurs puissances
$24$-\`emes sont des unit\'es dans des extensions ab\'eliennes de $K$.
 C'est ainsi que les propri\'et\'es
de la fonction $\eta$ et la th\'eorie de la multiplication
complexe permettent non seulement de d\'emontrer la conjecture de
Stark dans le cas o\`u $K$ est quadratique imaginaire, mais
apportent aussi une r\'eponse  \`a la Question
\ref{question:stark} dans ce cas.

Dans la g\'en\'eralisation ``traditionnelle" de la th\'eorie de la multiplication
complexe propos\'ee par Hilbert et son \'ecole, puis
d\'evelopp\'ee rigoureusement par Shimura et Taniyama, on
est amen\'e \`a
remplacer $\Q$ par un corps $F$ totalement r\'eel de degr\'e $n>1$
 (de sorte que
$F\otimes_\Q\R = \R^n$), et les formes modulaires
classiques par des {\em formes modulaires de Hilbert}.
Celles-ci correspondent \`a des fonctions holomorphes sur
$\cH^n$, invariantes (\`a un facteur d'automorphie pr\`es)
sous l'action naturelle de $\SL_2(\cO_F)$.
Les corps quadratiques imaginaires sont  remplac\'es par
des  extensions quadratiques $K$ de $F$ de type CM,
munies d'une identification
$\Phi:K\otimes_F \R^n \rightarrow \C^n.$
La th\'eorie de la multiplication complexe affirme alors
que les valeurs de certaines fonctions modulaires de Hilbert (quotients de formes
modulaires de m\^eme poids, poss\'edant des d\'eveloppements
de Fourier rationnels)
en des points  $\tau\in \Phi(K)\cap \cH^n$
 engendrent des extensions ab\'eliennes du corps
reflex  ${\tilde K}$  associ\'e \`a  $(K,\Phi)$.
Cette th\'eorie poss\`ede deux inconv\'enients du point de vue de la
Question \ref{question:stark}:
\begin{itemize}
\item[(a)]
elle ne permet d'aborder la ``th\'eorie du corps de classes explicite"
 que pour les corps de base de type
CM;
\item[(b)]
 elle ne permet pas d'obtenir  facilement
 des unit\'es dans des extensions
ab\'eliennes de  ${\tilde K}$,
 les unit\'es modulaires n'ayant pas
de g\'en\'eralisation \'evidente pour les formes modulaires de Hilbert.
En effet,  quand
$n>1$, le faisceau structural  sur l'espace complexe analytique
$\cX:= \SL_2(\cO_F)\backslash\cH^n$ ne poss\`ede pas de sections globales
non-nulles.
La relation entre la th\'eorie de Shimura-Taniyama et les conjectures
de Stark pour les corps CM (\`a supposer qu'il y en ait une) reste donc \`a
\'elucider. (Cf.~par exemple \cite{deshalit-goren} et
\cite{goren-lauter}.)
\end{itemize}

Pour aborder la Question \ref{question:stark} lorsque
 $K$ est une extension quadratique ATR d'un corps $F$ totalement
r\'eel de degr\'e $n>1$,
il faut  relever le nombre r\'eel $ \log |\tilde v_1(u_{a})|$ en un nombre
complexe $ \log \tilde v_1(u_{a})\in \C/2i\pi \Z.$
Au vu de la formule limite de Kronecker
(\ref{eq:LimKro}) et de la discussion  pr\'ec\'edente,
il s'agit donc de
g\'en\'eraliser l'expression
$$\log \eta(\tau) = \log |\eta(\tau)| + i \arg \eta(\tau) \in \C/2i\pi \Z $$ \`a un cadre
 o\`u les formes modulaires classiques sont
 remplac\'ees par des
formes modulaires de Hilbert sur $F$.
C'est ce qui a \'et\'e entrepris dans la th\`ese du premier auteur, qui part de l'identit\'e classique $\textrm{d}\log\eta(z)=-i\pi  E_2(z) dz$, o\`u $E_2$ est la s\'erie d'Eisenstein d\'efinie par
$$ E_2(z) = -\frac{1}{12} + 2 \sum_{n\ge 1} \sigma_1(n) e^{2i\pi n z},
\ \ \  \mbox{ avec }\, \sigma_1(n) = \sum_{d|n} d.$$

La formule (\ref{eqn:multeu}) peut donc se r\'e\'ecrire en prenant le
logarithme complexe des deux c\^ot\'es:
\begin{equation}
\label{eqn:logeu}
 \log(u(\tau_1,\tau_2))= -i\pi  \int_{\tau_1}^{\tau_2}E_2(z) dz \  \in  \C/2i\pi \Z.
\end{equation}
Or, si les unit\'es modulaires n'admettent pas d'analogue \'evident en
 dimension sup\'erieure, les s\'eries d'Eisenstein, elles, se
g\'en\'eralisent sans difficult\'e \`a ce contexte.
La section \ref{sec:eisenstein} rappelle la d\'efinition de la
 s\'erie d'Eisenstein $E_2$  de poids $(2,\ldots,2)$ sur
$\SL_2(\cO_F)$. Celle-ci donne lieu \`a  une $n$-forme
diff\'erentielle  $\omega_{E_2}$ holomorphe sur l'espace analytique
$\cX$.

La d\'emarche sugg\'er\'ee par \cite{charollois} consiste  essentiellement
 \`a remplacer les
valeurs de $\log \eta(\tau)$ par des int\'egrales  de $\omega_{E_2}$ sur des
cycles appropri\'es de dimension r\'eelle  $n$ sur $\cX$.
Plus pr\'ecis\'ement, le
 pr\'esent article associe \`a tout $\tau\in \cH\cap v_1(K)$
un cycle ferm\'e $\Delta_\tau$ de dimension r\'eelle $(n-1)$
sur $\cX$. En faisant abstraction pour le moment des ph\'enom\`enes li\'es
\`a la pr\'esence possible de torsion dans la cohomologie de $\cX$,
on d\'emontre que ces cycles sont homologues \`a z\'ero.  Autrement dit, il existe une
cha\^ine diff\'erentiable $C_\tau$
 de dimension $n$ sur $\cX$  telle que
$$ \partial C_\tau = \Delta_\tau.$$
Cela permet de d\'efinir un invariant canonique $J_\tau\in \C/2i\pi \Z$ en
int\'egrant un multiple appropri\'e de $\omega_{E_2}$ sur $C_\tau$.
La contribution la plus importante de cet article est
la Conjecture \ref{conj:principale}, qui
relie  les invariants $J_\tau$ \`a
l'expression $\log {\tilde v_1}(u_a)$ dont la partie r\'eelle appara\^{\i}t
dans la Conjecture \ref{conj:stark}. La Conjecture \ref{conj:principale}
apporte ainsi
un \'el\'ement de r\'eponse \`a la Question 1.

\medskip
La d\'efinition des invariants $J_\tau$
   s'appuie de fa{\c c}on essentielle sur
la th\`ese du premier auteur \cite{charollois}. Elle est 
 aussi \`a rapprocher de deux autres travaux ant\'erieurs:
\begin{enumerate}
\item L'article \cite{darmon-logan}, o\`u les s\'eries d'Eisenstein
sur $\GL_2(F)$ du pr\'esent article sont remplac\'ees par des
formes modulaires de Hilbert {\em cuspidales} de poids $(2,\ldots,2)$.
Dans les cas les plus
concrets qui ont pu \^etre test\'es num\'eriquement,
ces formes sont associ\'ees
\`a des courbes elliptiques
 d\'efinies sur $F$. L'invariant $J_\tau$ d\'efini  dans ce contexte
 semble alors permettre la construction
de points
 sur ces courbes d\'efinis sur certaines extensions ab\'eliennes de $K$.
\item
L'article \cite{darmon-dasgupta} peut  \^etre lu comme une variante $p$-adique
des constructions principales du pr\'esent article.
Dans le contexte de \cite{darmon-dasgupta},
on a $F=\Q$ et
le r\^ole de la place $v_1$  est  jou\'e par une place non-archim\'edienne $p$.
L'extension  $K$ est alors un corps quadratique
{\em r\'eel} dans lequel le nombre premier $p$ est inerte.
Les s\'eries d'Eisenstein de poids $2$ sur certains
sous-groupes de congruence de $\SL_2(\Z)$,
r\'einterpr\'et\'ees   convenablement
comme des ``formes modulaires de Hilbert" sur $\cH_p\times \cH$,
o\`u $\cH_p = \PP_1(\C_p)-\PP_1(\Q_p)$ est le demi-plan $p$-adique,
donnent alors lieu \`a des
invariants $p$-adiques  $J_\tau\in \C_p$  associ\'es \`a
$\tau\in\cH_p\cap K$.  Ces invariants
  correspondent conjecturalement \`a des $p$-unit\'es
dans des extensions ab\'eliennes de $K$.
\end{enumerate}
En se pla{\c{c}}ant dans un cadre  classique o\`u l'on dispose
 de notions topologiques et analytiques g\'en\'erales
  (homologie et cohomologie
singuli\`ere, th\'eorie de Hodge),  le pr\'esent article
m\`ene \`a une clarification conceptuelle des diff\'erentes constructions
de ``points et unit\'es de Stark-Heegner" propos\'ees jusqu'\`a pr\'esent
 dans la
litt\'erature.  (Cf. \cite{darmon-logan} et \cite{darmon-dasgupta}, ainsi que
le cadre  trait\'e originalement dans \cite{darmon1} ou
  les g\'en\'eralisations formul\'ees dans
\cite{trifkovic} et \cite{greenberg}.)
Le pr\'esent article peut tr\`es bien servir
d'introduction aux
travaux sur les points de  Stark-Heegner
cit\'es en r\'ef\'erence, bien qu'il ait \'et\'e
r\'edig\'e apr\`es ceux-ci.

\bigskip
\noindent
{\bf Remerciements}:
Le premier auteur tient \`a remercier
chaleureusement Philippe Cassou-Nogu\`es et Martin Taylor
qui sont  \`a l'origine de  son travail de th\`ese.
Les deux auteurs ont aussi b\'en\'efici\'e de nombreux \'echanges
avec Samit Dasgupta au sujet du pr\'esent article.

\section{Notions pr\'eliminaires}
\label{sec:preliminaires}
\subsection{S\'eries d'Eisenstein}
\label{sec:eisenstein}
Soit $F$ un corps totalement r\'eel de degr\'e $n>1$ et
$S_F:= \{ v_1,\ldots,v_n\}$
son ensemble de places archim\'ediennes. On désigne par $\mc O_F$ l'anneau des entiers de $F$,
par $d_F$ son discriminant,  et par $R_F$ le r\'egulateur de $F.$

On supposera dans  la suite   de cet article que $F$ a
  nombre de classes 1 au sens
restreint, de sorte qu'il existe pour tout $1\le j \le n $ une unit\'e
$\epsilon^{(j)}\in\cO_F^\times$ avec
$$ v_j(\epsilon^{(j)})<0,\qquad \qquad
v_k(\epsilon^{(j)})>0 \quad \mbox{ si } k\ne j. $$
Pour tout $a\in F$, on  note $a_j:=v_j(a)$ son image dans $\R$  par le plongement
$v_j$,
et  $A_j=\left(\begin{array}{cc} a_j&b_j\\
c_j&d_j
\end{array}\right)$
l'image
dans $\SL_2(\R)$
d'une
 matrice $A=\left(\begin{array}{cc} a&b\\
c&d
\end{array}\right)$ de $\SL_2(F)$.
On identifiera librement $a$ avec le $n$-uplet $(a_1,\ldots, a_n)$ et $A$ avec
$(A_1,\ldots , A_n)$.
On obtient ainsi une action   par homographies du
groupe modulaire de Hilbert $\Gamma:=\SL_2(\mc O_F)$
  sur le produit $\mc H_1\times \cdots\times \mc H_n$
 de $n$ copies du demi-plan
de Poincar\'e.
Le quotient analytique
$$ \cX:= \Gamma\backslash (\mc H_1\times\cdots\times  \mc H_n) $$
s'identifie avec les points complexes d'un ouvert de Zariski
 d'une vari\'et\'e alg\'ebrique
projective lisse: la {\em vari\'et\'e modulaire de Hilbert} $X_F$
associ\'ee au corps  $F$.

\smallskip
 Une \textit{forme modulaire de Hilbert} de poids
$(2,\ldots,2)$ pour $\Gamma$ est une fonction holomorphe $f(z_1,\ldots, z_n)$
sur $\mc H_1\times\cdots\times \mc H_n$ telle que la forme différentielle
$$\omega_f:= f(z_1,\ldots  z_n)\,dz_1\wedge\cdots\wedge dz_n $$
soit invariante sous l'action de  $\Gamma$.
Autrement dit, pour tout $\left(\begin{array}{cc} a & b \\ c & d
\end{array}\right)\in \Gamma$, on exige que
$$ f\left( \frac{a_1z_1+b_1}{c_1z_1+d_1},\ldots,
\frac{a_nz_n+b_n}{c_nz_n+d_n}\right)
= (c_1z_1+d_1)^2\cdots (c_n z_n+d_n)^2
f( z_1,\ldots,  z_n).$$
Lorsque  $n>1$,  une telle fonction possède, d'apr\`es le principe de Koecher,
un développement en série de
Fourier \`a l'infini de la forme
$$f(z_1,\ldots, z_n)=a_f(0)+ \sum_{\mu\in \mc O_F\!, \ \mu \gg0
}\!\!\!a_f(\mu)\,e^{2i\pi \big(\frac {\mu_1} {\delta_1}z_1 + \cdots +\frac
{\mu_n} {\delta_n}z_n\big)},$$
o\`u  $\delta$ d\'esigne un générateur totalement
positif de la différente de $F/\Q.$

 La \textit{série d'Eisenstein holomorphe
$E_2$} de poids $2$
se définit sur $\mc H_1\times\cdots\times \mc H_n$ par le
développement en série de Fourier suivant:
\begin{equation}
\label{devE2}
E_{2}(z_1,\ldots, z_n)=\zeta_F(-1)+
 2^n\! \! \! \sum_{\mu\in \mc O_F\!, \ \mu \gg0
}\!\!\!\sigma_{1}(\mu)\,e^{2i\pi \big(\frac {\mu_1} {\delta_1}z_1 +\cdots +
\frac {\mu_n} {\delta_n}z_n\big)},\end{equation}
o\`u, pour un
entier $k$ donné, on a posé
$$\sigma_k(\mu)=\sum_{(\nu)|(\mu)} |N_{F/\mb Q}(\nu)|^k,$$ la sommation
portant sur les idéaux (principaux) entiers $(\nu)$ qui divisent
$(\mu).$ La fonction $E_2(z_1,\ldots,z_n)$ est une forme modulaire de
Hilbert de poids $(2,\ldots,2)$ pour $\Gamma$ (voir [VdG], chap. I.6).

On se donne  des coordonn\'ees r\'eelles $x_j, y_j$ sur
$\cX$ en posant $z_j = x_j + i y_j$, et l'on
 d\'efinit \`a partir de $E_2$ une forme diff\'erentielle invariante
$\omega_{\Eis}$
en posant
\begin{equation}
\label{defomegaEis}
\omega_{\Eis} := \left\{\begin{array}{ll} \frac{(2i\pi)^2}{\sqrt{d_F}}
\omega_{E_2}+ \frac{R_F}{2}( \frac {dz_1\wedge d\bar{z}_1}{y_1^2} -\frac {dz_2\wedge d\bar{z}_2}{y_2^2})\ & \textrm{ si } n=2,\\
\frac{(2i\pi)^n}{\sqrt{d_F}} \omega_{E_2} \ & \textrm{ si } n\geq 3.\end{array} \right.
\end{equation}
La $n$-forme diff\'erentielle $\omega_{\Eis}$ est ferm\'ee et, quand $n\geq 3$,  elle est holomorphe.
On va s'int\'eresser \`a sa classe dans la cohomologie
$H^n(\cX,\C)$
form\'ee \`a partir du complexe de deRham des formes
diff\'erentielles $C^\infty$ sur $\cX$.
Ce groupe de cohomologie est muni d'une action des op\'erateurs de
Hecke $T_\lambda$, o\`u les $\lambda$  parcourent les
id\'eaux de $\cO_F$. La forme diff\'erentielle
$\omega_{\Eis}$ est vecteur propre pour ces op\'erateurs.
Plus pr\'ecis\'ement, on a
$$ T_\lambda(\omega_{\Eis}) =  \sigma_1(\lambda)\omega_{\Eis}.$$
Pour tout $1\le j\le n$,
on dispose \'egalement
d'une involution $T_{v_j}$
sur l'espace r\'eel-analytique
$\mc X$
associ\'ee \`a la place $v_j$
(``op\'erateur de Hecke \`a l'infini"). Celle-ci se
d\'efinit
en posant
$$ T_{v_j}(z_1, \ldots, z_n) :=
 (\epsilon^{(j)}_1z_1, \ldots, \epsilon^{(j)}_j\bar{z}_j, \ldots,
\epsilon^{(j)}_nz_n).
$$
On appelle $T_{v_j}^*$ l'involution sur $H^n(\cX,\C)$ qui
s'en d\'eduit par ``pullback" sur les formes diff\'erentielles.
Les $n$ op\'erateurs $T_{v_j}^*$  et
les
  op\'erateurs de Hecke $T_\lambda$ engendrent  une alg\`ebre commutative  $\TT$ sur $\Z.$

On aura besoin dans la suite de certaines fonctions qui joueront le
r\^ole de ``primitives" de $E_2$. On introduit pour cela
 la fonction $h$ de Asai
\cite{Asai}   d\'efinie sur $\mc H^n$ par
\begin{equation}
\label{devhAsai}
h(z)=\frac{4(-\pi)^n \zeta_F(-1)}{R_F\sqrt{d_F}}
 y_1\cdots y_n
+ \frac{4\sqrt{d_F}}{2^{n}R_F}
\sum_{\mu\in
\, \mc O_F}\prim \!\!\!\!\!\!\!\sigma_{-1}(\mu) \, e^{2i\pi \left(\frac {\mu_1}
{\delta_1} x_1+\left|\frac{\mu_1}{\delta_1}\right|iy_1+\ldots+\frac {\mu_n}
{\delta_n} x_n+\left|\frac{\mu_n}{\delta_n}\right|iy_n\right)}.
\end{equation}
Il sera plus  commode de  travailler avec
$${\tilde h}(z):= \lambda_F h(z), \quad \mbox{ o\`u }
 \lambda_F:=  4^{n-1} R_F,$$
de sorte que
\begin{equation}
\label{devhAsaiprime}
{\tilde h}(z)=\frac{(-4\pi)^n \zeta_F(-1)}{\sqrt{d_F}}
 y_1\cdots y_n
+ 2^{n} \sqrt{d_F}
\sum_{\mu\in
\, \mc O_F}\prim \!\!\!\!\!\!\!\sigma_{-1}(\mu) \, e^{2i\pi\left(\frac {\mu_1}
{\delta_1} x_1+\left|\frac{\mu_1}{\delta_1}\right|iy_1+\ldots+\frac {\mu_n}
{\delta_n} x_n+\left|\frac{\mu_n}{\delta_n}\right|iy_n\right)}.
\end{equation}
Les fonctions $h(z)$ et ${\tilde h}(z)$ jouissent des propri\'et\'es suivantes:
\begin{enumerate}
\item
Elles sont harmoniques par rapport
\`a chacune des  variables $z_j$,  $1\leq j\leq n$;
\item
Elles satisfont les formules de transformation
\begin{eqnarray}
\label{eqn:transh}
 h(Az) &=&  h(z) - \log(|c_1 z_1+d_1|^2 \cdots |c_n z_n + d_n|^2), \\
\label{eqn:transhtilde}
 \tilde h(Az) &=&  \tilde h(z) - \lambda_F \log(|c_1 z_1+d_1|^2 \cdots |c_n z_n + d_n|^2).
\end{eqnarray}
\item
On a
\begin{eqnarray}
\label{eqn:dh}
\frac{\partial^n \tilde h(z_1,\ldots,z_n)}{\partial z_1\cdots \partial z_n} d z_1\wedge \cdots \wedge dz_n
&=&  \frac{(2i\pi)^n}{\sqrt{d_F}}\omega_{E_2}, \\
\label{eqn:dhbar}
\frac{\partial^n \tilde h(z_1,\ldots,z_n)}{\partial{z_1} \cdots\partial{\bar{z_j}}\cdots  \partial z_n} d {z_1} \wedge \cdots d\bar{z}_j\cdots \wedge dz_n
&=&  T_{v_j}^*\left(\frac{(2i\pi)^n}{\sqrt{d_F}}\omega_{E_2}\right).
\end{eqnarray}
\end{enumerate}
Toutes ces formules
se v\'erifient par un calcul direct,
sauf (\ref{eqn:transh}) et (\ref{eqn:transhtilde}).
Pour ces derni\`eres, voir  \cite{Asai},
Th\'eor\`eme 4.
\begin{lemma}
\label{premierlemme}
\label{lemmeRe_j(Eis)}
La forme $(\Id+T^*_{v_j})\omega_{\Eis}$
est exacte.
\end{lemma}
\begin{proof}
Supposons que $j=1$ pour fixer les id\'ees et
all\'eger les notations.
\`A partir de la fonction  ${\tilde h}$, on d\'efinit la $(n-1)$ forme
diff\'erentielle sur $\cH^n$:
\begin{equation}
\label{eqn:defeta}
\eta=\frac{\partial^{n-1}\tilde h(z_1,\ldots,z_n)}
{\partial z_2 \cdots \partial z_n } dz_2\wedge \ldots \wedge dz_n.
\end{equation}
Quand $n>2$, la formule (\ref{eqn:transhtilde})  montre que
$\eta$ est  invariante sous $\Gamma,$ et  correspond donc \`a une $(n-1)$-forme
diff\'erentielle sur $\cX$.
Parce que ${\tilde h}$ est harmonique,
cette forme
 est de plus holomorphe par rapport aux variables
 $z_2,\ldots,z_n,$
d'o\`u la formule
\begin{equation}
\label{deta}
d\eta = 
\left(\frac{\partial^{n} \tilde h}{ \partial z_1 \cdots \partial z_n }dz_1\wedge \ldots\wedge  dz_n +
\frac{\partial^{n} \tilde h}{ \partial \bar{z_1}
\partial z_2\cdots \partial z_n}d\bar{z_1}\wedge dz_2\ldots \wedge dz_n\right).
\end{equation}
Le lemme  r\'esulte alors de
(\ref{eqn:dh}) et de (\ref{eqn:dhbar})  avec $j=1$.
\smallskip

Dans le cas o\`u $n=2$, la forme
$\eta$ d\'efinie par
(\ref{eqn:defeta})
n'est plus $\Gamma$-invariante.
En effet, si $A=\mat{a}{b}{c}{d}\in\Gamma$, on a
$$ A^*(\eta) = \eta - 4R_F\frac{c_2}{c_2 z_2 + d_2} dz_2. $$
Il convient alors de modifier la d\'efinition de $\eta$ en (\ref{eqn:defeta})
en posant cette fois
\begin{equation}
\label{eqn:defeta2} \eta':=\left(\frac{\partial  \tilde
h(z_1,z_2)}{ \partial z_2 \ }- \frac{2R_F}{iy_2}\right)dz_2.
\end{equation}
On d\'eduit de l'identit\'e (\ref{eqn:transhtilde}) que
 $\eta'$ est invariante sous $\Gamma.$
La formule (\ref{deta})  s'adapte sans
difficult\'e  \`a condition d'ajouter  la contribution de
$$d\left(\frac{-2R_F}{i y_2} dz_2\right)= -R_F\frac {dz_2\wedge d\bar{z}_2}{y_2^2}.$$
On  obtient
$$ (\Id+T^*_{v_1})\omega_{\Eis} = d {\eta'},$$
 puisque  la forme  $dz_1\wedge d\bar{z}_1/y_1^2$ est quant
\`a elle dans le noyau de $\Id+T_{v_1}^*.$
C'est ce calcul qui justifie le terme suppl\'ementaire
apparaissant  dans la d\'efinition (\ref{defomegaEis})  de
 $\omega_{\Eis}$ lorsque $n=2$.
\end{proof}

Pour $m\ge 0$,
on appelle $C_m^0(\mc X)$ le groupe  engendr\'e par les combinaisons lin\'eaires
formelles \`a coefficients dans $\Z$ des cycles
diff\'erentiables ferm\'es de dimension
r\'eelle $m$ sur $\cX$.
On d\'efinit le groupe
des p\'eriodes de $\omega_{\Eis}$ par
$$ \Lambda_{\Eis}:= \left\{ \int_C \omega_{\Eis} \quad \mbox{ pour } C
 \in C_n^0(\mc X) \right\}
 \subset \C.$$

\begin{proposition}
\label{propositionreseauperiode}
Le groupe $\Lambda_{\Eis}$  est un sous-groupe de $\C$  de rang un,
commensurable
avec $(2i\pi )^n \Z$.
\end{proposition}
\begin{proof}
Cette proposition se d\'emontre en trois parties.

\sk\noindent
(a) On d\'emontre d'abord que
 le groupe $\Lambda_{\Eis}$ {\em est un sous-ensemble discret de
$\C$.}

La th\'eorie de Harder
\cite{harder} (cf.~le Th\'eor\`eme 6.3, Ch. III, \S 7 de \cite{freitag} avec $m=n$) fournit  la d\'ecomposition
\begin{equation}
\label{eqn:freitag}
H^n(\cX,\C) = H^n_{\rm univ}(\cX,\C) \oplus H^n_{\Eis}(\cX,\C) \oplus
H^n_{\rm cusp}(\cX,\C),
\end{equation}
o\`u $H^n_{\rm univ}(\cX,\C)$ provient des formes diff\'erentielles
$\SL_2(\R)^n$-invariantes,
$H^n_{\Eis}(\cX,\C)$  est l'espace vectoriel de dimension
$1$ engendr\'e par  $[\omega_{\Eis}]$,
et
$H^n_{\rm cusp}(\cX,\C)$ provient des formes modulaires cuspidales
de poids $(2,\ldots,2)$ sur $\cX$.
La d\'ecomposition (\ref{eqn:freitag})
 est respect\'ee par l'alg\`ebre de Hecke $\TT$,
et  les \'el\'ements  $\omega\in H^n_{\Eis}$ sont caract\'eris\'es
par les propri\'et\'es
$$ T_{v_j}^*\omega = -\omega, \qquad
T_\lambda \omega = (\N\lambda+1)\omega, \quad \mbox{ pour tout } \lambda\lhd \cO_F.$$
Il en r\'esulte que la projection
 naturelle
 $H^n(\cX,\C)\rightarrow H^n_{\Eis}(\cX,\C)$ issue de
(\ref{eqn:freitag}) est d\'ecrite par
  un idempotent $\pi_{\Eis}\in\TT\otimes\Q$.
Soit
 $\Lambda$ l'image naturelle de $H_n(\cX,\Z)$ dans
le dual $H^n(\cX,\C)^\vee:= \Hom(H^n(\cX,\C),\C)$
 par l'application des p\'eriodes.
C'est un sous-groupe discret  stable pour l'action de $\TT$.
On a de plus
$$ \Lambda_{\Eis} = \langle \omega_{\Eis}, \pi_{\Eis}(\Lambda)\rangle,$$
o\`u $\langle\ ,\ \rangle$ d\'esigne l'accouplement
naturel entre $H^n(\cX,\C)$ et son dual.
Or on a
$$\pi_{\Eis}(\Lambda) \subset \frac{1}{t}\Lambda \cap H^n_{\Eis}(\cX,\C)^\vee,$$
o\`u $t$ est  un entier tel que $t\pi_{\Eis}$ appartient \`a $\TT$.
Par cons\'equent  $\pi_{\Eis}(\Lambda)$ est un sous-groupe discret
de $H^n_{\Eis}(\cX,\C)^\vee$, ce qui implique que $\Lambda_{\Eis}$ est lui aussi discret.

\sk\noindent
(b) {\em Le groupe $\Lambda_{\Eis}$ est contenu dans $(2i\pi)^n\R$.}

En effet,
le Lemme \ref{premierlemme} implique que
$$ T_{v_1}^* \cdots T_{v_n}^* ([\omega_{\Eis}]) = (-1)^n [\omega_{\Eis}].$$
Par ailleurs, un calcul direct montre que
$$ T_{v_1}^* \cdots T_{v_n}^* ([\omega_{\Eis}]) = [\bar \omega_{\Eis}].$$
On en d\'eduit que $[\bar\omega_{\Eis}] = (-1)^n [\omega_{\Eis}]$.
Les p\'eriodes de $\omega_{\Eis}$ appartiennent donc bien
\`a $(2i\pi)^n \R$.

\sk\noindent
(c) {\em Fin de la d\'emonstration}

Les parties (a) et (b)
montrent que $\Lambda_{\Eis}$ est de rang au plus un.
Pour montrer qu'il est non-trivial et d\'eterminer sa classe de commensurabilit\'e,
il suffit de
calculer une p\'eriode non-nulle de $\omega_{\Eis}$.
Pour cela, on fixe $Y_1,\ldots, Y_n\in \R^{>0}$ et l'on consid\`ere
les droites horizontales $L_j\subset \cH_j$
form\'ees des $z_j$ dont la partie imaginaire est \'egale \`a $Y_j$.
La r\'egion $$ R_\infty:= L_1\times\cdots \times L_n$$
est pr\'eserv\'ee par le sous-groupe des translations $\Gamma_\infty\subset \Gamma.$ Soit $D_\infty$ un domaine fondamental compact pour cette
action. Son image dans $\cX$ est un cycle ferm\'e de dimension $n$. La forme diff\'erentielle
$\omega_{\rm Eis}$ peut s'int\'egrer terme \`a terme sur
$D_\infty$
\`a partir de la formule (\ref{devE2}). Seul le terme constant
dans la d\'efinition de $E_2$ apporte une contribution
non-nulle
\`a l'int\'egrale, puisque les autres termes sont des multiples de
caract\`eres non-triviaux de $R_\infty/\Gamma_\infty$.
Comme le volume de $R_\infty/\Gamma_\infty$ est \'egal \`a $\sqrt{d_F}$,
on en d\'eduit que
$$ \int_{D_\infty} \omega_{\Eis} = (2i\pi)^n \zeta_F(-1).$$
Ceci ach\`eve la d\'emonstration, puisque $\zeta_F(-1)$
appartient \`a $\Q^\times$.
\end{proof}

\subsection{Extensions quadratiques et s\'eries $L$ associ\'ees}\label{subsec:quad}
Soit $K$
 une extension quadratique de $F$.
Pour chaque $1\le j \le n$, la $\R$-alg\`ebre
 $K\otimes_{F,v_j} \R$  est isomorphe soit \`a $\C$, soit \`a
$\R\oplus \R$.
On fixe une telle identification, que l'on appelle aussi $v_j$ par abus de
notation.
Lorsque $K$ est un corps CM, la donn\'ee de $(v_1,\ldots, v_n)$ correspond
au choix d'un type CM associ\'e \`a $K$. On sait \`a quel point
cette donn\'ee
suppl\'ementaire joue un r\^ole important dans la th\'eorie de la multiplication
complexe pour les extensions CM de $F$.

On munit les $\R$-espaces vectoriels $\C$ et $\R\oplus \R$ de l'orientation
standard dans
 laquelle  une orientation positive est assign\'ee
aux bases $(1,i)$ et $((1,0),(0,1))$ de $\C$ et $\R\oplus \R$  respectivement.
Une base de $K$ (vu comme espace vectoriel sur $F$ de dimension $2$) est
alors dite {\em positive} si ses images dans $\C$ et $\R\oplus\R$ par les
plongements $v_j$  sont  orient\'ees positivement.
On remarque en particulier que la base $(1,\tau)$ de $K$ sur $F$ est
positive si et seulement si:
\begin{enumerate}
\item  On a $\tau_j'>\tau_j$ pour toute place r\'eelle $v_j$ de $F$;
\item  La partie imaginaire de $\tau_j$ est strictement
positive pour toute place complexe
$v_j$.
\end{enumerate}

Soit $I$ un id\'eal de l'anneau $\cO_F$. On se permet de d\'esigner  par le
m\^eme symbole l'id\'eal $I\cO_K$ de $K$.

On maintiendra tout au long de cet article l'hypoth\`ese
 que $F$ a pour nombre de classes $1$ au sens \'etroit.
Par contre, il  est souhaitable
 de ne pas avoir \`a faire d'hypoth\`ese semblable
sur le corps $K$.
On g\'en\'eralise la d\'efinition (\ref{eqn:defGIintro})
du groupe $\cG_I$ de l'introduction, en le d\'efinissant comme un
quotient appropri\'e du groupe
$\A_K^\times$  des id\`eles de $K$.
Pour chaque place
non-archim\'edienne $v$ de $K$, on appelle $\cO_v$ l'anneau des entiers du
corps local $K_v$, et l'on pose
$$ \cG_I:=  \A_K^\times/(\prod_v U_v) K^\times,$$
avec
$$ U_v = \left\{  \begin{array}{ll}
\R_{>0}^\times & \mbox{ si } v \mbox{ est r\'eelle}; \\
\C^\times & \mbox{ si } v \mbox{ est complexe}; \\
\cO_v^\times & \mbox{ si } v\nmid I; \\
1 + I \cO_v & \mbox{ si } v|I. \end{array} \right.
$$
Comme dans l'introduction, la
 loi de r\'eciprocit\'e du corps de classes donne un isomorphisme
$\rec:\cG_I\lra \gal(H/K)$, o\`u $H$ est le corps de classes de rayon de
$K$ au sens restreint associ\'e \`a $I$.

Le sous-corps $F$ permet d'introduire un sous-groupe
$\cG_I^+\subset \cG_I$, d\'efini comme
 l'image naturelle dans $\cG_I$ du groupe
$\A_F^\times$.
Le sous-corps $H_I$
de $H$ fix\'e par $\rec(\cG_I^+)$ s'appelle le
{\em corps de classes d'anneau} associ\'e \`a $I$ et $K/F$.
On a donc l'isomorphisme de r\'eciprocit\'e
$$ \rec: G_I \lra  \Gal(H_I/K), \quad\mbox{o\`u   }
G_I:= \A_K^\times/(\A_F^\times \prod_v U_v K^\times).$$

Un $\cO_F$-ordre de $K$ est un sous-anneau de $K$ qui contient
$\cO_F$ et qui est localement
libre de rang $2$ sur $\cO_F$ (donc libre, puisque $h(F)=1$).
On d\'esigne par
 $\cO_I:= \cO_F + I \cO_K$ l'ordre de $K$ de conducteur $I$,
et l'on appelle
$$\hat\cO_I = \prod_{v\nmid\infty} \cO_I\otimes \cO_{F,v}$$
 son ad\'elisation.
On a alors
$$ G_I = \A_K^\times/(\hat\cO_I^\times \prod_{v|\infty} U_v K^\times).$$
Ce quotient est en bijection naturelle avec le groupe
$\Pic^+(\cO_I)$  des modules projectifs de rang 1 sur $\cO_I$ dans $K$, modulo
l'\'equivalence au sens restreint.
(Deux modules projectifs $M_1$ et $M_2$
 sur $\cO_I$ sont dits \'equivalents au sens restreint s'il
existe  un \'el\'ement totalement positif
$k\in K_+^\times$ tel que $M_2=k M_1$).
On  associe en effet \`a tout $\alpha\in G_I$ un  $\cO_I$-module  
$M\subset K$ en posant
$$ M^\alpha:= \alpha \hat{\cO}_I \cap K. $$
L'application $\alpha\mapsto M^\alpha$ fournit une bijection
naturelle entre $G_I$ et les classes d'\'equivalence au sens restreint
de $\cO_I$-modules projectifs:
\begin{equation}
\label{eqn:giproj}
G_I \, \simarrow
 \left\{
\begin{array}{l}
\mbox{Classes d'\'equivalence} \\
\mbox{au sens restreint} \\
\mbox{de $\cO_I$-modules projectifs}
\end{array} \right\}.
\end{equation}

Soit $V:= \cO_{I,1}^\times\subset \cO_I^\times$  le groupe des unit\'es de
$\cO_I$ de norme $1$ sur $\Q$.
Il
laisse stable le module $M$.
On a la suite exacte
$$ 0\lra V_1 \lra V\lra \cO_{F,1}^\times,$$ o\`u   $V_1$  d\'esigne le sous-groupe des
unit\'es de $V$ de norme relative  $1$ sur $F$.  On note alors  ${\tilde V}$  le sous-groupe de $V$ engendr\'e par
$V_1$ et $\mc O_F^\times$,
et l'on pose
$$ \delta_I := [V:{\tilde V}].$$
Cet indice est un diviseur de $2^n$.
 On d\'efinit finalement la fonction $L(M, s)$ associ\'ee
\`a un $\cO_I$-module projectif $M$ en posant
\begin{equation}
\label{eqn:LMs}
L(M,s) :=   (\N M)^s \delta_I\!\! \sum_{x\in M/V}
\prim\!\!\!\!\textrm{signe}(N_{K/\Q}(x))
 | \N_{K/\Q}(x)|^{-s}.
\end{equation}
\begin{remark}

a) 

Parce que $\sign(N_{K/\Q}(x))=1$ quand
$K$ est quadratique imaginaire, 
il en r\'esulte que
la d\'efinition (\ref{eqn:LMs})
g\'en\'eralise l'\'equation
(\ref{eqn:Llattice}) de l'introduction.

b)  Quand $K$ est une extension ATR de $F$
ayant $v_1$ pour unique place complexe, la fonction $L(M,s)$ est un multiple rationnel
non-nul de la  somme de fonctions $L$ partielles de Hurwitz de l'introduction.
Plus pr\'ecis\'ement,
si $a$ est un g\'en\'erateur de $\cO_K/(\cO_F,I)$ en tant que
$(\cO_F/I)$-module, et que $M_a$ d\'esigne le $\cO_I$-module projectif
$$ M_a := \{ x\in\cO_K \quad \mbox{tel qu'il existe } r\in \cO_F \mbox{ avec }
x \equiv r a \pmod{I} \},$$
alors
\begin{equation}
\label{eqn:LlatticeF}
L(M_a,s) = \delta_I \!\! \sum_{r\in(\cO_F/I\cO_F)}
L(ra,I,s).
\end{equation}
\end{remark}

Comme dans l'introduction,
on v\'erifie  que la fonction $L(M,s)$ ne d\'epend que de la classe d'\'equivalence
de $M$ au sens restreint, puisque
$$ L(\lambda M,s) =\textrm{signe}(N_{K/\Q}(\lambda))
  L(M,s).$$

Dans les Sections \ref{sec:TR} et \ref{sec:ATR} qui suivent, nous allons  exprimer les valeurs
sp\'eciales $L(M,0)$ quand $K$ est totalement r\'eel, et les d\'eriv\'ees $L'(M,0)$
quand $K$ est ATR, en fonction de p\'eriodes appropri\'ees de la forme diff\'erentielle
$\omega_{\Eis}$.

\section{Extensions quadratiques  totalement
r\'eelles et val\-eurs de fonctions $L$}
\label{sec:TR}
On supposera dans cette section que l'extension quadratique
 $K$ de $F$ est totalement r\'eelle.
On veut rappeler un th\'eor\`eme qui appara\^it dans la th\`ese du premier auteur et qui
donne une formule explicite pour $L(M,0)$ dans ce cas.

L'hypoth\`ese que $F$ a nombre de classes $1$ au sens restreint
implique que le module $M$ est libre de rang $2$ comme module sur
$\cO_F$, et qu'il existe une $\cO_F$-base 
$(\omega_1,\omega_2)$   de $M.$
On suppose que cette base est choisie de sorte que $(1,\omega_2/\omega_1)$ 
soit orient\'ee positivement. 
 L'invariant
 $\tau:=\omega_2/\omega_1$
 appartient 
\`a $K^\times$, et il ne d\'epend que de la classe
d'\'equivalence de $M$ au sens restreint, \`a l'action de $\Gamma$ pr\`es.
Le groupe $\Gamma_\tau\subset \Gamma$ form\'e des matrices qui
fixent $\tau$ est un groupe
de rang $n$ (modulo torsion), que
l'application
$$ A = \mat{a}{b}{c}{d} \mapsto c\tau + d$$
identifie
 avec le sous-groupe $V_1$ des
unit\'es de $V$ de norme relative  $1$ sur $F$. Pour chaque $1\le j\le n$, on pose
$$ (\tau_j,\tau_j'):=v_j(\tau) \in \R\times\R,$$
et l'on appelle
 $\Upsilon_j$ la g\'eod\'esique hyperbolique sur $\cH_j$
joignant $\tau_j'$ \`a $\tau_j$, orient\'ee dans le sens
allant de $\tau_j'$ \`a $\tau_j$. Le produit
$$ R_\tau = \Upsilon_1 \times \Upsilon_2\times\cdots\times\Upsilon_n\subset \mc H ^n$$
est un espace contractile hom\'eomorphe \`a $\R^{n}$.
On le munit de l'orientation naturelle h\'erit\'ee des
$\Upsilon_j$.  Le groupe $\Gamma_\tau$ op\`ere
sur $R_\tau$ par transformations
de M\"obius, et le quotient $\Gamma_\tau\backslash R_\tau$
est compact,
isomorphe \`a un tore r\'eel de dimension $n$.
Soit $\Delta_\tau$ un domaine fondamental pour l'action de
$\Gamma_\tau$ sur $R_\tau$.
On identifie $\Delta_\tau$ avec son image dans $\cX$, qui est un cycle ferm\'e
dans ce quotient.
\begin{theorem}\label{thm:realmain}
Pour tout $\cO_I$-module projectif  $M$ dans $K,$ on a:
\begin{equation}
\label{lienLomegatotreel}
(-2)^n \int_{\Delta_\tau} \omega_{\Eis} =  (2i\pi)^n L(M, 0).
\end{equation}
\end{theorem}
\begin{proof}
C'est une cons\'equence du corollaire \ref{cor:final} qui est
d\'emontr\'e dans la derni\`ere partie de cet article.
 Puisque $K$ est une extension quadratique totalement
r\'eelle de $F,$ on choisit $r=n\geq 2$ et $c=0$ dans la formule
(\ref{L0r2}). Elle s'\'ecrit alors
\begin{equation}
L(M, 0)=\frac{i^n}{\pi^n}\int_{\Delta_\tau}
\frac{\partial^n \tilde h(z_1,\ldots,z_n)}{\partial z_1\cdots \partial z_n} d z_1 \cdots dz_n.
\end{equation}
D'apr\`es  l'identit\'e  (\ref{eqn:dh}), on  en d\'eduit  que
\begin{equation}
L(M,
0)=\frac{i^n}{\pi^n}\int_{\Delta_\tau}\frac{(2i\pi)^n}{\sqrt{d_F}}\omega_{E_2}.
\end{equation}
 La formule  (\ref{lienLomegatotreel}) en r\'esulte imm\'ediatement lorsque $n>2$ vu la d\'efinition (\ref{defomegaEis})
de $\omega_{\textrm{Eis}}.$ Cette formule reste encore  valable
pour $n=2$ puisque les int\'egrales des formes $dz_1\wedge
d\bar{z}_1/y_1^2$ et $dz_2\wedge d\bar{z}_2/y_2^2$ sur le cycle
$\Delta_\tau$   sont nulles.
\end{proof}

\bigskip

\begin{corollary}
Pour tout r\'eseau $M$ dans $K$, les valeurs sp\'eciales
$L(M,0)$ sont rationnelles. Plus pr\'ecis\'ement, il existe une constante enti\`ere $e_F,$ ne   d\'ependant  que du corps $F$ et pas de l'extension $K$ ni de $M,$  telle que $e_FL(M, 0)\in \Z.$ \end{corollary}
\begin{proof}
Cela r\'esulte de ce que les p\'eriodes de
$\omega_{\Eis}$, d'apr\`es la
Proposition  \ref{propositionreseauperiode},  appartiennent \`a un
r\'eseau $\Lambda_{\Eis}\subset (2i\pi)^n \Q$ qui ne d\'epend que du corps $F.$
\end{proof}

\section{Extensions quadratiques  ATR et d\'eriv\'ees de fonctions $L$}
\label{sec:ATR}
\noindent

On suppose dans cette section  que l'extension quadratique
 $K$ de $F$ est ATR, et que $v_1$ se prolonge en une place complexe de $K$.
On veut
donner dans ce cas une formule explicite pour $L'(M, 0),$ lorsque
$M$ est un $\cO_I$-module projectif dans $K$.

Comme dans la section pr\'ec\'edente,  on pose
$\tau:= \omega_2/\omega_1$,
o\`u $(\omega_1,\omega_2)$ est une $\cO_F$-base positive de $M$,
 choisie de  sorte
que $(1,\tau)$ soit positivement orient\'ee. 
On pose ensuite
$$ \left\{
\begin{array}{lll}
\tau_1:= v_1(\tau) & \in \cH_1 &  \\
(\tau_j,\tau_j'):= v_j(\tau) & \in \R\times \R, &  \mbox{ pour } j=2,\ldots,n.
\end{array} \right.
$$
Le nombre complexe  $\tau_1$ appartient alors
\`a $\cH_1$.  Pour
 chaque $2\le j\le n$,
 on appelle
 $\Upsilon_j$ la g\'eod\'esique hyperbolique de  $\cH_j$
joignant $\tau_j$ \`a $\tau_j'$, orient\'ee dans le sens
allant de $\tau_j$ \`a $\tau_j'$. Le produit
$$ R_\tau = \{\tau_1\} \times \Upsilon_2\times\cdots\times\Upsilon_n\subset \mc H^n$$
est un espace contractile hom\'eomorphe \`a $\R^{n-1}$, que l'on
 munit de l'orientation naturelle h\'erit\'ee des
$\Upsilon_j$.
Le stabilisateur $\Gamma_\tau$ de  $\tau$
dans $\Gamma$ est  un groupe
de rang $(n-1)$ modulo torsion, que
l'on peut identifier avec le sous-groupe d'unit\'es relatives  $V_1$
introduit pr\'ec\'edemment.
Il op\`ere
 sur $R_\tau$ par transformations
de M\"obius, et le  quotient $\Gamma_\tau\backslash R_\tau$ est compact,
isomorphe \`a un tore r\'eel de dimension $(n-1)$.
Soit $\Delta_\tau$ un domaine fondamental pour
 l'action de $\Gamma_\tau$ sur $R_\tau$.
On identifie $\Delta_\tau$ avec son image
dans $\cX$, qui est un cycle ferm\'e de dimension
$(n-1)$ dans ce quotient.
\begin{lemma}\label{lemmaDeltatrivial}
La classe de $\Delta_\tau$ dans $H_{n-1}(\cX,\Z)$ est de torsion.
En particulier, il existe une
$n$-cha\^ine diff\'erentiable  $C_\tau$ \`a coefficients dans $\Q$ telle que
\begin{equation}\label{defCtau}\partial C_\tau = \Delta_\tau.\end{equation}
\end{lemma}
\begin{proof}
Le sous-groupe de torsion de $H_{n-1}(\cX,\Z)$ s'identifie avec le
noyau de l'application naturelle $H_{n-1}(\cX,\Z)\lra
H_{n-1}(\cX,\Q)$. Lorsque $n$ est pair, le Lemme
\ref{lemmaDeltatrivial} r\'esulte de ce que  $H_{n-1}(\cX,\Q)=0$ (cf.
\cite{freitag}, Ch.~III).  De m\^eme,
lorsque $n=2m+1$ est impair, le groupe $H^{n-1}(\mc X,\C)$ est
engendr\'e par les classes des $\left( \begin{array}{l} n \\  m
\end{array}\right)$ formes diff\'erentielles du type
$$ \eta_S:= \bigwedge_{j\in S} \frac{dz_j \wedge d\bar z_j}{y_j^2}, \qquad S\subset \{1,\ldots,n\}, \quad
\#S = m.$$
Or on voit que les restrictions de ces classes sur
$\Delta_\tau$ (et m\^eme sur les r\'egions $R_\tau$) sont nulles,
puisque
la projection de $R_\tau$ sur chaque facteur $\mc H_j$ est de dimension r\'eelle
 $0$ ou $1$.
On en d\'eduit par la dualit\'e de Poincar\'e
 que l'image de $\Delta_\tau$ dans $H_{n-1}(\cX,\C)$ est nulle.
\end{proof}

On introduit  $\omega_{\rm Eis}^{+}=\frac{ 1}{2}(\Id+T_{v_1}^*)\omega_{\Eis}$
la projection de la forme diff\'erentielle
$\omega_{\Eis}$ sur l'espace propre de
$T_{v_1}^*$ associ\'e \`a la valeur
propre $+1,$ autrement dit la ``partie r\'eelle
pour la place $v_1$" de la forme $\omega_{\rm Eis}.$
\begin{theorem}
\label{thm:atrmain}
Soit
$M$ un $\cO_I$-module projectif associ\'e \`a $\tau\in K$.
L'int\'egrale de $\omega_{\Eis}^{+}$ sur $C_\tau$ ne d\'epend pas du choix de $C_\tau$ v\'erifiant (\ref{defCtau}),
et l'on a
\begin{equation}
\label{eqn:Lstarkperiode}
(-2)^{n-1} \int_{C_\tau} \omega_{\mathrm{Eis}}^{+}  = (2i\pi)^{n-1}L'(M,0).
\end{equation}
\end{theorem}
\begin{proof}
La premi\`ere assertion d\'ecoule du fait que $\omega_{\Eis}^{+}$
est exacte:
le Lemme \ref{lemmeRe_j(Eis)} montre que $\omega_{\Eis}^{+}=d\eta/2.$ Le calcul se poursuit en
utilisant   le th\'eor\`eme de Stokes pour obtenir
\begin{equation}\label{eq:Stokes1}\int_{C_\tau} \omega_{\Eis}^{+}=\int_{\Delta_\tau} \frac \eta 2.\end{equation}
Supposons tout d'abord que $n>2,$ de sorte que $\eta$ est  la
($n-1$)-forme holomorphe  sur  $\cX$ donn\'ee  par la formule
(\ref{eqn:defeta}). Comme $K$ est une extension ATR, on choisit
ici $r=n-1\geq 2$ et $c=1$  dans le corollaire \ref{cor:final}.i).
La formule correspondante s'\'ecrit
\begin{equation}
\label{LM0primeintermediaire} L'(M,0)= \frac{i^{n-1}}{2\pi^{n-1}}
\int_{\Delta_\tau}  \eta,
\end{equation}
 ce qui permet de conclure.

 \smallskip
Il reste \`a traiter   le cas o\`u $n=2$
 en faisant cette fois appel au corollaire \ref{cor:final}.ii).
 On choisit $r=1$ et $c=1$ dans la formule (\ref{L0r1}) qui
 s'\'ecrit

 \begin{equation*}\label{L'M0intermediaire}L'(M,0)=\frac{i}{2\pi}
\int_{\Delta_\tau} \left( \frac {\partial \tilde{h}}{\partial
z_2}-\frac{4R_F} {z_2-\bar{z}_2}\right) dz_2.
\end{equation*}
 Au vu de la d\'efinition (\ref{eqn:defeta2}) de la forme  $\eta$ pour $n=2,$ cette   \'egalit\'e se r\'eduit \`a   $$\int_{\Delta_\tau} -  \eta= (2i\pi )L'(M, 0).$$
 La formule de Stokes permet de nouveau  de conclure \`a la formule souhait\'ee.
 \end{proof}

\section{Application d'Abel-Jacobi et unit\'es de Stark}
\label{sec:conjecture}

Quand on combine le Th\'eor\`eme \ref{thm:atrmain} avec la Conjecture \ref{conj:stark}  de Stark,
on obtient la formule conjecturale suivante pour le   logarithme
du module de l'unit\'e de Stark:
\begin{equation}
\label{eqn:suggestive}
 e_I(-2)^{n-1}\int_{C_\tau} \omega_{\rm Eis}^+ = 2\delta_I
(2i\pi)^{n-1} \log|{\tilde v_1}(u_\tau)|,
\end{equation}
o\`u $e_I$ d\'esigne l'ordre du groupe des racines de l'unit\'e dans  $H_I.$
Pour relever
l'invariant r\'eel  $\log|{\tilde v_1}(u_\tau)|=\textrm{Re }(\log {\tilde v_1}(u_\tau))$
 en un invariant complexe bien d\'efini modulo $2i\pi \Z$,
il suffira de remplacer dans
la formule (\ref{eqn:suggestive})
la diff\'erentielle exacte $\omega_{\Eis}^+$ par la forme diff\'erentielle
$\omega_{\Eis}.$

Pour tout $m\ge 0$, on
d\'esigne par $C_m(\cX)$ le groupe  engendr\'e par les combinaisons lin\'eaires
formelles \`a coefficients dans $\Z$ des cha\^ines diff\'erentiables de dimension
r\'eelle $m$
sur $\cX$, et l'on d\'esigne par $C_m^0(\cX)$
et $C_m^{00}(\cX)$  les sous-groupes engendr\'es par
les cycles diff\'erentiables ferm\'es et homologues \`a z\'ero
respectivement:
\begin{eqnarray*}
C_m^0(\cX) &:=& \{\Delta\in C_m(\cX) \mbox{ tel que }
 \partial \Delta  = 0\}.\\
C_m^{00}(\cX) &:=& \{\Delta\in C_m(\cX) \mbox{ tel qu'il existe }
C\in C_{m+1}(\cX) \mbox{ avec } \partial C = \Delta\}.
\end{eqnarray*}
On pose  aussi
$$
{\tilde C}_m^{00}(\cX)
:=  \{\Delta\in C_m(\cX) \mbox{ tel qu'il existe }
C\in C_{m+1}(\cX)\otimes\Q \mbox{ avec } \partial C = \Delta\}.
$$

On sait que $C_{n-1}^0(\cX)/C_{n-1}^{00}(\cX) = H_{n-1}(\cX,\Z)$
 est un groupe de type  fini, dont le sous-groupe de torsion s'identifie
avec ${\tilde C}_{n-1}^{00}(\cX)/C_{n-1}^{00}(\cX)$.
Soit $n_F$  l'exposant  de ce groupe fini, et soit
$$ {\Lambda'}_{\Eis}:= \frac{1}{n_F}\Lambda_{\Eis}.$$
 On peut d\'efinir \`a partir de la forme
diff\'erentielle
$\omega_{\Eis}$  une ``application d'Abel-Jacobi"
$$ \Phi_{\Eis}: C_{n-1}^{00}(\cX)\lra \C/\Lambda_{\Eis}$$
en posant
$$
\Phi_{\Eis}(\Delta) = \int_{\partial C = \Delta} \!\!\! \omega_{\Eis} \pmod{\Lambda_{\Eis}},$$
l'int\'egrale \'etant prise sur n'importe quelle $n$-cha\^ine diff\'erentiable $C$
sur $\cX$  tel que $\partial C=\Delta$.
Cette application $\Phi_{\Eis}$ est bien d\'efinie
modulo le r\'eseau des p\'eriodes $\Lambda_{\Eis}$ en vertu
de la Proposition \ref{propositionreseauperiode}.
Quitte \`a remplacer le r\'eseau $\Lambda_{\Eis}$ par
$\Lambda_{\Eis}'$, on peut \'etendre $\Phi_{\Eis}$ au groupe ${\tilde C}_{n-1}^{00}(\cX)$
tout entier, en posant
\begin{equation}\label{def:PhiEisprolonge}
\Phi_{\Eis}(\Delta) = \frac{1}{n_F} \int_{\partial C = n_F\Delta} \!\!\! \omega_{\Eis}
\pmod{\Lambda'_{\Eis}},\end{equation}
l'int\'egrale \'etant prise sur n'importe quelle $n$-cha\^ ine diff\'erentiable $C$
sur $\cX$  tel que $\partial C= n_F\Delta$.
On  pose ensuite
$$ J_\tau:= e_I(-2)^{n-1} \Phi_{\Eis}(\Delta_\tau) \in \C/\Lambda_{\Eis}'.$$

Soit $\Lambda_{\Eis}''$ le r\'eseau de $(2i\pi)^n\R$ engendr\'e
par $\Lambda_{\Eis}'$ et $(2i\pi)^n\Z$.
On fixe une place $\tilde v_1$ de $H_I$ au-dessus de la place $v_1$ de $K$.
Nous sommes maintenant en mesure d'\'enoncer la conjecture principale de cet
article.
\begin{conjecture}
\label{conj:principale}
Pour tout $\cO_I$-module $M$ associ\'e \`a $\tau\in K$,  il
existe une unit\'e $u_\tau\in \cO_{H_I}^\times$
telle que
$$J_\tau = 2\delta_I (2i\pi)^{n-1} \log(\tilde v_1(u_{\tau}))
 \pmod{\Lambda_{\Eis}''}.$$
De plus, pour
 tout $2\le j \le n$, l'image de $u_\tau$ par
n'importe quel plongement complexe au-dessus de $v_j$
est de module $1$.
 Pour tout $\alpha\in G_I$, on a $u_{\tau^\alpha} = \rec(\alpha)^{-1} u_\tau$,
o\`u $\tau^\alpha$ d\'esigne l'invariant associ\'e au module $M^\alpha$.
\end{conjecture}

\section{Algorithmes}
\label{sec:algorithmes}
L'invariant $J_\tau$ et
l'application d'Abel-Jacobi $\Phi_{\Eis}$
ont l'inconv\'enient de  ne pas
\^etre faciles \`a calculer num\'eriquement a priori.
Le but de la pr\'esente section est de d\'ecrire un algorithme pour
le calcul de $\Phi_{\Eis}$ dans la cas le plus simple o\`u
$n=2$.

La premi\`ere \'etape consiste \`a   d\'ecrire
la classe de cohomologie de $\omega_{\Eis}$ en terme
de cohomologie du groupe $\Gamma$.

On rappelle
 le dictionnaire bien connu entre la cohomologie de deRham de
$\cX$ et la cohomologie de $\Gamma$.
Si $P,Q,R$ sont des points de $\cH_1\times\cH_2 = \cH^2$, on appelle
$\Delta(P,Q,R)$
n'importe quelle
$2$-cha\^ine  diff\'erentiable
dont  la fronti\`ere est \'egale au
 triangle g\'eod\'esique de sommets
$P$, $Q$ et $R$.
On munit cette
r\'egion  de l'orientation standard,
selon les d\'efinitions usuelles de l'homologie singuli\`ere.
  On pose aussi, pour $P =(z_1,z_2)\in\cH^2$ et
$A,B\in \Gamma$,
$$ \Delta_P(A,B):= \Delta(P,A P,AB P).$$
On associe \`a $\omega_{\Eis}$
(plus pr\'ecis\'ement: \`a sa classe de cohomologie) un $2$-cocycle
$$ \kappa_P\in \mc Z^2(\Gamma,\C)$$
par la r\`egle
$$ \kappa_P(A,B):= \int_{\Delta_P(A,B)} \omega_{\Eis}.$$
Un calcul direct montre que $\kappa_P$ satisfait la relation de
$2$-cocycle: $d\kappa_P = 0$, et que son image  dans
$H^2(\Gamma,\C)$ ne d\'epend pas du choix du point base $P$.

On rappelle le r\'eseau $\Lambda_{\Eis}\subset \C$ des p\'eriodes
de $\omega_{\Eis}$ et
l'on note $\bar\kappa_P$ l'image de $\kappa_P$ dans
$\mc Z^2(\Gamma,\C/\Lambda_{\Eis}')$.

\begin{lemma}
\label{lemma:rhop}
La classe de $\bar\kappa_P$ dans
 $H^2(\Gamma,\C/\Lambda_{\Eis}')$ est nulle.
\end{lemma}
\begin{proof}
Pour tout $A\in\Gamma$, on appelle $S_P(A)$  l'image dans $\cX$ du chemin
g\'eod\'esique  sur
$\cH^2$
allant de $P$ \`a $AP$.
Comme $S_P(A)$ est un $1$-cycle ferm\'e sur $\cX$ et que $H_1(\cX,\Q)=0$,
il existe une $2$-cha\^ine diff\'erentiable sur $\cX$ \`a coefficients
entiers, que l'on appellera
$D_P(A)$, telle que
\begin{equation}\label{defDA} \partial D_P(A) = n_F S_P(A).\end{equation}
La r\'egion $D_P(A)$ est d\'etermin\'ee par cette \'equation
 modulo les $2$-cycles ferm\'es, et par con\-s\'e\-quent
l'\'el\'ement de $\C/\Lambda_{\Eis}'$ d\'efini par
\begin{equation}
\label{eqn:rhotheoretical}
 \rho_P(A):= \frac{1}{n_F}\int_{D_P(A)} \omega_{\Eis} \pmod{\Lambda_{\Eis}'}
\end{equation}
ne d\'epend pas du choix de $D_P(A)$ satisfaisant (\ref{defDA}).
On v\'erifie ensuite par un calcul direct que
$$d\rho_P(A,B) = \kappa_P(A,B) \pmod{\Lambda_{\Eis}'}.$$
\end{proof}

Le Lemme \ref{lemma:rhop} permet de d\'efinir une $1$-cha\^ine $\rho_P$
en choisissant une solution de l'\'equation
\begin{equation}
\label{eqn:rhop}
d\rho_P = \kappa_P \pmod{\Lambda_{\Eis}'}.
\end{equation}

La proposition suivante permet de calculer l'invariant
num\'erique  $\Phi_{\Eis}(\Delta_\tau)$
en terme de cohomologie
des groupes---du moins, en admettant
que l'on sache r\'esoudre l'\'equation
(\ref{eqn:rhop}).

Soit $K$ un corps ATR et soit
  $\tau\in K$ un \'el\'ement provenant
 d'une base positive d'un r\'eseau $M\subset K$.
Parce que $n=2$, le groupe $\Gamma_\tau$ est de rang un modulo la torsion.
On se donne un g\'en\'erateur $\gamma_\tau$ de $\Gamma_\tau $ modulo torsion,
choisi  de sorte  que  pour tout point $z_2$ de la g\'eod\'esique $\Upsilon_2$, le
chemin allant de $z_2$ \`a $\gamma_\tau z_2$ soit orient\'e dans le sens positif.
On choisit le point base $P\in\cH_1\times\cH_2$
de mani\`ere \`a ce que sa premi\`ere composante
 soit \'egale \`a $\tau_1=v_1(\tau)$.
Avec ces choix, on a alors
\begin{proposition}
$$ \Phi_{\Eis}(\Delta_\tau) = \rho_P(\gamma_\tau).$$
\end{proposition}
\begin{proof}
Cela r\'esulte directement de la formule pour
$\rho_P$ de l'\'equation (\ref{eqn:rhotheoretical}).
\end{proof}

La d\'efinition du $2$-cocycle $\kappa_P$ exige d'int\'egrer
$\omega_{\Eis}$ sur des r\'egions de type  $\Delta_P(A,B)$ peu commodes \`a param\'etrer.
Dans les calculs num\'eriques, il est donc utile  de  remplacer ce cocycle par un
 repr\'esentant de la m\^eme classe de cohomologie  qui ne fait intervenir
que des r\'egions  ``rectangulaires"
de  la forme $L_1\times L_2\subset\cH_1\times \cH_2$ (avec $L_1$ et $L_2$
de dimension $1$, bien entendu). Les int\'egrales de $\omega_{\Eis}$ sur de telles r\'egions
s'expriment au moyen d'int\'egrales it\'er\'ees,
et sont  donc plus faciles \`a
calculer num\'eriquement. (On se sert pour cela
 du d\'eveloppement de  Fourier de $\omega_{\Eis}$.)

Si $u,v$ appartiennent \`a $\cH$,
soit $\Upsilon[u,v] \subset \cH$ le segment g\'eod\'esique joignant
le point $u$ au point $v$.  On pose
$$ \Box_P(A,B) = \Upsilon[z_1,A_1 z_1] \times \Upsilon[A_2 z_2,A_2B_2 z_2],$$
et on d\'efinit un nouveau cocycle
$\kappa^\Box_P\in \mc Z^2(\Gamma, \C)$  par la r\`egle
$$\kappa^\Box_P(A, B)=\int_{\Box_P(A,B)}\!\!\omega_{\Eis}.$$
Il  est n\'ecessaire de modifier l\'eg\`erement  $\kappa_P^\Box$ pour qu'il repr\'esente  la m\^eme
classe de cohomologie que $\kappa_P.$
On dispose pour cela  d'un
 $2$-cocycle classique sur $\SL_2(\R)$  appelé \textit{cocycle d'aire}, dont on rappelle la d\'efinition : \'etant donn\'ees deux matrices  $M=\left(\begin{array}{cc} * &*\\
c& d
\end{array}\right)$ et $N=\left(\begin{array}{cc} * &*\\ c'& d'
\end{array}\right)$ de $\SL_2(\R)$ et  en notant   $MN=\left(\begin{array}{cc} * &*\\ c''&
d''
\end{array}\right)$ leur produit, la formule
\begin{equation}\label{definitioncocycleaire}
\textrm{ aire}(M,\,N):=-\textrm{signe}(c\,c'\, c'')\end{equation}
(où $\textrm{signe}(x)=x/|x|$ si $x\neq 0,$ et  $0$ sinon) définit
un  $2$-cocycle  sur $\SL_2(\R)$ à
valeurs enti\`eres. Par composition avec les  plongements de $\Gamma$ dans $\SL_2(\R),$ on en
 d\'eduit deux $2$-cocycles  sur $\Gamma$ \`a valeurs dans $\Z.$
On d\'efinit finalement le $2$-cocycle
$\tilde\kappa_P$ sur $\Gamma$  par la formule
\begin{align*}\tilde\kappa_P(A,B):&= \kappa_P^\Box(A,B)   - i\pi R_F\, \textrm{aire}(A_1,B_1)+i \pi R_F\,\textrm{aire}(A_2,B_2)  \\  &\int_{z_1}^{A_1 z_1}\!\!\!\int_{A_2
z_2}^{A_2B_2 z_2} \!\!\!\!\omega_{\Eis}\  -i \pi R_F \,\textrm{aire}(A_1,B_1)+ i\pi R_F\, \textrm{aire}(A_2,B_2).\end{align*}
\begin{proposition}
\label{prop:calgary-mtl}
Les cocycles $\kappa_P$ et $\tilde\kappa_P$ repr\'esentent la m\^eme classe
de cohomologie dans $H^2(\Gamma, \C).$  Plus pr\'ecis\'ement, on a
$$ \kappa_P(A,B) - \tilde\kappa_P(A,B) = d\xi_P(A,B),$$
o\`u
$$ \xi_P(A) = -\int_{\Delta_P(A)} \omega_{\Eis}, \quad
\mbox{ avec } \  \Delta_P(A) = \Delta((z_1,z_2),(z_1,A_2z_2), (A_1z_1,A_2z_2)).$$
\end{proposition}
\begin{proof}
On dit que deux $2$-cha\^ines $Z_1$ et $Z_2$ sont homologues si leurs fronti\`eres sont \'egales,
et on \'ecrit dans ce cas $Z_1\sim Z_2$.
Un calcul direct  fournit la relation
\begin{equation}\label{formuleavionCalgary}
-\Box_P(A,B)  + \Delta_P(A,B) + \Delta_P(A) -\Delta_P(AB) \sim \Delta_1+\Delta_2-\Delta_3,\end{equation}
avec  $$\left\{\begin{array}{l}\Delta_1=\Delta((A_1B_1z_1,A_2B_2 z_2),(z_1,A_2B_2z_2),(A_1 z_1,A_2B_2z_2)),\\
\Delta_2= \Delta((z_1,z_2),(z_1,A_2 z_2),(z_1,A_2B_2z_2)),\\
\Delta_3= \Delta((A_1z_1,A_2z_2), (A_1z_1,A_2B_2z_2), (A_1B_1z_1,A_2B_2z_2)).\end{array}\right.$$
Par $A$-invariance, on observe que $\Delta_3=\Delta_P(B)$ dans $\mc X.$
En outre,  l'int\'egrale de $\omega_{E_2}$ sur $\Delta_1$ et $\Delta_2$ est nulle, et par cons\'equent
$$\int_{\Delta_1+\Delta_2} \omega_{\textrm{Eis}} = \frac{R_F} 2 \int_{\Delta_1}
\frac {dz_1\wedge d\bar{z}_1}{y_1^2}-\frac{R_F} 2\int_{\Delta_2}\frac {dz_2\wedge d\bar{z}_2}{y_2^2}.$$
Ces derni\`eres int\'egrales se calculent \'el\'ementairement  :  on constate  d'abord
qu'elles ne d\'epen\-dent pas du point base $P=(z_1, z_2),$ et que $dz_j\wedge d\bar{z}_j=-2idx_j\wedge dy_j.$ Or
 l'int\'egrale
$$\int_{\Delta_j}
\frac {dx_j\wedge dy_j}{y_j^2}$$ n'est rien d'autre que l'aire, dans le disque de Poincar\'e, du triangle id\'eal orient\'e de sommets
$\infty,$ $A_j\infty$ et $A_jB_j\infty.$
D'apr\`es  \cite{K-M} formule 1.2, il en r\'esulte que
$$\int_{\Delta_1}
\frac {dz_1\wedge d\bar{z}_1}{y_1^2}=-2i\pi \,\textrm{aire}(A_1,B_1) \ \   \textrm{ et }\int_{\Delta_2}
\frac {dz_2\wedge d\bar{z}_2}{y_2^2}=-2i\pi  \,\textrm{aire}(A_2,B_2).$$
On conclut alors de (\ref{formuleavionCalgary}) que $$\kappa_P(A, B)=\kappa_P^\Box(A, B)+d\xi_P(A, B)-i\pi R_F  \, \textrm{aire}(A_1,B_1)+ i \pi R_F  \,\textrm{aire}(A_2,B_2),$$
d'o\`u la proposition.
\end{proof}

\begin{corollary}
\label{cor:calgary-mtl}
Soit $\tilde\rho_P$ une solution de l'\'equation
\begin{equation}
\label{eqn:rhoptilde}
 d\tilde\rho_P = \tilde \kappa_P\pmod{\Lambda_{\Eis}'}.
\end{equation}
Alors on a $\Phi_{\Eis}(\Delta_\tau) = \tilde\rho_P(\gamma_\tau).$
\end{corollary}
\begin{proof}
La Proposition \ref{prop:calgary-mtl}
montre que l'on peut choisir $$\tilde\rho_P(A) = \rho_P(A) -\xi_P(A) \pmod{\Lambda_{\Eis}'},$$
o\`u la $1$-cocha\^ine $\xi_P$ est d\'efinie dans l'\'enonc\'e de cette proposition.
Comme la r\'egion $\Delta_P(\gamma_\tau)$ qui intervient dans la
formule pour $\xi_P(\gamma_\tau)$ est contenue dans
$\{\tau_1\}\times \Upsilon[z_2,\gamma_\tau z_2]$,
 on a $$ \xi_P(\gamma_\tau) = 0.$$
Le corollaire en r\'esulte.
\end{proof}

\begin{remark}
Dans le pr\'esent article le cocycle $\tilde\kappa_P$ n'intervient que dans
les algorithmes pour calculer $J_\tau$ num\'eriquement.
Signalons tout de m\^eme que
la proposition \ref{prop:calgary-mtl} et le corollaire
\ref{cor:calgary-mtl}
sont d'un int\'er\^et plus que pratique.
Dans le contexte partiellement $p$-adique  \'etudi\'e
dans \cite{darmon-hpxh} et
\cite{darmon-dasgupta} o\`u l'on  est amen\'e \`a
travailler avec des formes
modulaires sur $\cH_p\times\cH$, on ignore comment donner un sens aux r\'egions
de la forme $\Delta_P(A,B)$,  ou au cocycle $\kappa_P$.
Par contre,  on sait
 d\'efinir  ce qui
 doit jouer le r\^ole des int\'egrales
``it\'er\'ees"
de formes modulaires (cuspidales ou Eisenstein) sur des r\'egions ``rectangulaires"
de la forme $\Box_P(A,B)$. Cela
 permet de d\'efinir un avatar $p$-adique
de $\tilde\kappa_P$, et par cons\'equent des versions $p$-adiques
des invariants $J_\tau$ du pr\'esent article.
\end{remark}

Il reste finalement  \`a calculer une solution de l'\'equation
(\ref{eqn:rhop}) ou
(\ref{eqn:rhoptilde}).   Le proc\'ed\'e \'etant le m\^eme, qu'il s'agisse
de $\rho_P$ ou de $\tilde\rho_P$, on se bornera au cas de
$\rho_P$ pour all\'eger les notations.

L'algorithme que nous proposons pour
 calculer $\rho_P(\gamma)$, pour $\gamma$ n'importe quel \'el\'ement de $\Gamma$,
 se base sur l'observation suivante: lorsque
 $\gamma= hkh^{-1}k^{-1}$ est un commutateur dans $\Gamma$, la formule (\ref{eqn:rhop})
permet d'exprimer  $\rho_P(\gamma)$ directement en
fonction de   $\kappa_P$.
En effet, l'identit\'e facile  $\rho_P(\Id) = 0$ assure  que
 $$\rho_P(h)+\rho_P(h^{-1}) =\kappa_P(h,h^{-1}).$$
En reportant, on en conclut  que
\begin{eqnarray*}
\rho_P(\gamma) &=& -\kappa_P(h,kh^{-1}k^{-1})-\kappa_P(k,h^{-1}k^{-1})- \kappa_P(h^{-1},k^{-1}) \\
   & & +\kappa_P(h,h^{-1})+\kappa_P(k,k^{-1})\pmod{\Lambda'_{\Eis}}.
\end{eqnarray*}
Cette derni\`ere formule, avec $\rho_P$ et $\kappa_P$
remplac\'es par $\tilde\rho_P$ et $\tilde\kappa_P$ respectivement,
donne un acc\`es num\'erique \`a
$\tilde\rho_P(hkh^{-1}k^{-1})$ puisque les nombres complexes $\tilde \kappa_P(g,g')$
se calculent gr\^ace au  d\'eveloppement en s\'erie de Fourier de la s\'erie d'Eisenstein.

\medskip
Enfin, l'ab\'elianis\'e $\Gamma_\textrm{ab}$ de $\Gamma$ est fini
(voir [DL, Prop. 1.3]). Son ordre  divise $4N_{F/\mb Q}(\epsilon^2-1),$
o\`u $\epsilon$ d\'esigne l'unit\'e fondamentale de $F.$
Pour calculer $\rho_P(\gamma)$ pour une matrice $\gamma$ de $\Gamma,$
il suffit  donc de  d\'ecomposer $\gamma^{|\Gamma_{\textrm{ab}}|}$ en un produit de commutateurs.

\smallskip
Sous l'hypoth\`ese que $F$ est de nombre de classes $1,$ on peut proc\'eder comme suit.
L'anneau des entiers $\mc O_F$ est  euclidien en $k$-\'etapes pour la norme  selon la terminologie de Cooke [Co, Th. 1]. Par cons\'equent, le groupe modulaire de Hilbert $\Gamma$ est engendr\'e par les matrices
\'el\'ementaires de type  suivant:
l'involution $S\left(\begin{array}{cc}0&-1\\1&0\end{array}\right),$ les matrices de translation $T_\theta=\left(\begin{array}{cc}1& \theta\\0 & 1\end{array}\right),$ et les puissances de la matrice $U=\left(\begin{array}{cc}\epsilon & 0\\ 0 & \epsilon^{-1} \end{array}\right).$
Il en r\'esulte que   $\gamma^{|\Gamma_{\textrm{ab}}|}$ s'\'ecrit comme un produit de matrices \'el\'ementaires gr\^ace \`a l'algorithme d'Euclide dans $\cO_F,$ puis comme un produit de commutateurs \`a l'aide des relations

$$UT_\theta U^{-1}T_{\theta}^{-1}=T_{\theta(\epsilon^2-1)},\hspace{0.5cm} SUS^{-1}U^{-1}=U^2.$$

\medskip
\begin{remark} Dans notre contexte ``Eisenstein", on ne peut pas utiliser
tel quel l'algorithme propos\'e dans
[DL,  section 4]. En effet, les int\'egrales du type
$\int^\tau\int ^{c_1}_{c_2} \omega_f$ (avec $c_1,$ $c_2 \in \PP^1(F)$)
n'ont de sens que si $f$ est une forme modulaire de Hilbert cuspidale.
Notons cependant que cet algorithme et  le n\^otre reposent
tous  deux sur l'hypoth\`ese que $\mc O_F$ est un anneau euclidien.
\end{remark}

\section{Exemples numériques}
Dans cette partie, nous pr\'esentons quelques r\'esultats exp\'erimentaux
 obtenus gr\^ace \`a l'algorithme pr\'ec\'edent.
Il s'agit, pour quelques cas d'extensions ATR $K/F$,  de tester
nu\-m\'e\-ri\-que\-ment la Conjecture
\ref{conj:principale} de cet article  et   d'exhiber
le polyn\^ome minimal de l'unit\'e  attendue.

\medskip
\subsection{Corps de base $\Q(\sqrt{5})$}

On consid\`ere d'abord la situation o\`u $F=\Q(\sqrt{5})$ et l'on
 note $\epsilon=\frac{1+\sqrt{5}}2$ son unit\'e fondamentale de norme $-1.$
L'anneau des entiers $\mc O_F=\Z[\epsilon]$ est euclidien pour la norme.
On fixe les places archim\'ediennes $v_1$ et $v_2$ de $F$ de  sorte    que
 $\epsilon_1<0$ et $\epsilon_2>0.$ On supposera dans cette section que $\Lambda''_{\rm Eis}=\Lambda'_{\rm Eis},$ et l'on se donne un entier   $m_F>0$  tel que $\Lambda'_{\Eis}\subset (2i\pi)^2m_F\Z.$  Les exemples ci-dessous laissent penser que   $m_{\Q(\sqrt 5)}=15$ convient.

Nous \'etudions maintenant les invariants associ\'es \`a diff\'erentes extensions quadratiques ATR  $K$ de $F$  dans lesquelles la place $v_1$  devient complexe.

\medskip
\noindent
{\bf (a) Un exemple \`a groupe des classes  $C_4$.}

\medskip

On consid\`ere
$K=F(\sqrt{21\epsilon-11})$. C'est une extension ATR de $F$, dont le
groupe des  classes au sens restreint
est  cyclique  d'ordre $4.$

\medskip
Aux quatre classes  distinctes
$\cC_1,\ldots,\cC_4$ de $\cO_K$
  au sens restreint,
on associe les \'el\'ements  $\tau_1, \ldots,\tau_4$ de $K$  fix\'es
par les  matrices suivantes de $\Gamma$:

$$\begin{array}{ll}\gamma_1\left(\begin{array}{cc}4\epsilon + 2 & -2\epsilon - 5\\
-2\epsilon - 1 & 2\epsilon + 1\end{array}\right), &
\gamma_2=\left(\begin{array}{cc}13\epsilon + 9 & 4\epsilon + 1\\ -32\epsilon - 18& -7\epsilon -
6\end{array}\right), \\

\\
 \gamma_3=\left(\begin{array}{cc} -47\epsilon - 17 & 9\epsilon
- 6\\ -520\epsilon - 308& 53\epsilon + 20\end{array}\right),
 &
\gamma_4= \left(\begin{array}{cc}165\epsilon + 79& 5\epsilon - 2\\ -8512\epsilon - 5160&
-159\epsilon - 76\end{array}\right).
\end{array}$$
On calcule dans $\C/m_F \Z$ les invariants $\rho_k(\gamma_k):=\tilde{\rho}_{\tau_k}(\gamma_k)/(2i\pi)^2$ associ\'es.
On trouve  avec une pr\'ecision minimale de 50 d\'ecimales significatives
\begin{eqnarray*}
\rho_1(\gamma_1)&\approx& -0.3666666666\ldots-0.27784944302\ldots i;\\
\rho_2(\gamma_2)&\approx&0.32623940638\ldots; \\
\rho_3(\gamma_3) &\approx& 1.83333333333\ldots+0.27784944302\ldots i;\\
\rho_4(\gamma_4) &\approx& 17.8404272602\ldots.
\end{eqnarray*}
Sans conna\^itre la constante $m_F,$ on doit tester l'alg\'ebricit\'e du nombre complexe bien d\'efini
$$u_k(m_F)=\exp(2i\pi m_F\rho_k(\gamma_k)).$$
Cependant, dans la  pratique, il semble qu'il existe toujours
 une  racine $m_F$-i\`eme de $u_k(m_F)$ qui appartienne au corps de d\'efinition de  $u_k(m_F).$
Pour  chaque valeur de   $\rho_k(\gamma_k)$ dans la liste pr\'ec\'edente,
notons $$u_k(1)=\exp(2i\pi\rho_k(\gamma_k)).$$
Le nombre complexe $u_k(1)$ est bien d\'efini seulement modulo les racines $m_F$-i\`emes de l'unit\'e.
 Quitte \`a modifier $u_k(1)$ par une racine de l'unit\'e, on peut donc
esp\'erer tester avec succ\`es son alg\'ebricit\'e. Pr\'ecis\'ement, la commande Pari
\texttt{algdep($u_1(1)e^{-\frac 4 {15}i\pi },16)$}  sugg\`ere
la relation alg\'ebrique suivante pour
 $u_1:=u_1(1)e^{-\frac 4 {15}i\pi }$:
\begin{equation}
Q_1(x):=x^8+4x^7-10x^6+x^5+9x^4+x^3-10x^2+4x+1.
\end{equation}
On en conclut que le  nombre complexe $u_1$
co\"incide sur 50 d\'ecimales avec la racine $-5.7303\ldots $
du polyn\^ome $Q_1$.
Il en va de m\^eme pour les trois autres invariants  : \\
$u_2:=u_2(1)e^{\frac {23}{15}i\pi}$ co\"incide
avec la racine $0.834403847893\ldots+0.5511535345\ldots i$ de $Q_1.$\\
$u_3:=u_3(1)e^{\frac{20}{15}i\pi}$ co\"incide avec la racine  $-0.17450889906\ldots=1/u_1.$\\
$u_4:=u_4(1)e^{\frac2{15}i\pi}$ co\"incide avec la racine $0.834403847893\ldots-0.5511535345\ldots i.$

On  v\'erifie  a posteriori  que $Q_1$ est effectivement  le polyn\^ome minimal
d'une unit\'e du corps de classes de Hilbert (au sens restreint) de $K.$

\medskip

\noindent
{\bf (b) Un exemple \`a groupe des classes $C_6$.}

\medskip

On consid\`ere  maintenant $K= F(\sqrt{26\epsilon-37})$,
dont le  nombre de classes au sens restreint
 est $6.$
On trouve avec parfois 200 d\'ecimales de pr\'ecision dans $\C/m_F \Z$:
\begin{eqnarray*}
\rho_1(\gamma_1) &\approx&   4.499999999999999\ldots - 0.728584512\ldots i; \\
\rho_2(\gamma_2) &\approx&  -1.078476376302846\ldots- 0.195385083050863\ldots i; \\
\rho_3(\gamma_3) &\approx&    -2.178476376302846\ldots  + 0.195385083050863\ldots i;\\
\rho_4(\gamma_4) &\approx&  -18.61666666666666\ldots+ 0.728584510266413\ldots i; \\
\rho_5(\gamma_5) &\approx&  -0.988190290363819\ldots+ 0.195385083050863\ldots i; \\
\rho_6(\gamma_6) &\approx&  -2.421523623697153\ldots - 0.195385083050863\ldots i.
\end{eqnarray*}

\noindent
La commande   \texttt{algdep($u_2(1)e^{\frac {16} {15}i\pi},12)$}  de Pari sugg\`ere la relation alg\'ebrique:
\begin{eqnarray*}
Q_2(x) &=& x^{12}+106x^{11}+873x^{10}-2636x^9+3040x^8-626x^7-1108x^6-626x^5 \\
   & & +3040x^4+2636x^3+873x^2+106x+1.
\end{eqnarray*}
Ce polyn\^ome est effectivement le polyn\^ome minimal d'une unit\'e de $H_{K}^+.$ En outre,
les nombres complexes
\begin{equation*}\begin{array}{llll}
u_1:&=u_1(1) &=& -97.30316237461782\ldots,\\
u_2:&=u_2(1)e^{\frac{16}{15}i\pi}&\approx&-3.276785825745970\ldots + 0.955188763599790\ldots i, \\
u_3:&=u_3(1)e^{\frac{19}{15}i\pi}& \approx& -0.281276149057161\ldots  +  0.081992486337387\ldots i, \\
u_4:&=u_4(1)e^{\frac{5}{15}i\pi} &\approx& -0.010277158271074\ldots, \\
u_5:&=u_5(1)e^{\frac{16}{15}i\pi} &\approx&  -0.281276149057161\ldots - 0.081992486337387\ldots i, \\
u_6:&=u_6(1)e^{\frac{16}{15}i\pi} &\approx& -3.276785825745970\ldots  - 0.955188763599790\ldots i
 \end{array}
\end{equation*}
co\"incident chacun avec une racine de $Q_2$ sur plusieurs dizaines de d\'ecimales.
\medskip

\noindent
{\bf (c) Un exemple \`a groupe des classes  $C_2\times C_4$.}

\medskip

Le groupe des classes (au sens restreint) de $K=F(\sqrt{21\epsilon-29})$ est d'ordre $8,$ isomorphe \`a $C_2\times C_4.$ L'algorithme d\'ecrit pr\'ec\'edemment permet de calculer
\begin{eqnarray*}
\rho_1(\gamma_1) &\approx&     -1.866666666666\ldots - 0.787374943777\ldots i, \\
\rho_2(\gamma_2) &\approx&      0.297896510457\ldots +0.068709821260\ldots i, \\
\rho_3(\gamma_3) &\approx&  -0.300000000000\ldots+ 0.161542382812\ldots i, \\
\rho_4(\gamma_4) &\approx&  -1.097896510457\ldots + 0.068709821260\ldots i, \\
\rho_5(\gamma_5) &\approx&    -0.133333333333\ldots + 0.787374943777\ldots i, \\
\rho_6(\gamma_6) &\approx&  -0.031229843791\ldots- 0.068709821260\ldots i, \\
\rho_7(\gamma_7) &\approx&  -0.900000000000\ldots - 0.161542382812\ldots i, \\
\rho_8(\gamma_8) &\approx& -1.38121\ldots - 0.391304\ldots i.
\end{eqnarray*}

\noindent
L'invariant le plus pr\'ecis est  $\rho_5(\gamma_5)$ dont on a obtenu  200 d\'ecimales significatives.
La commande   Pari
 \texttt{algdep($u_5(1)e^{-\frac {19} {15}i\pi},16, 200)$}
   fournit comme candidat  le polyn\^ome
   r\'eciproque
\begin{eqnarray*}
Q_3(x) &=& x^{16}+139x^{15}-255x^{14}-538x ^{13}+2018x^{12}-2237x^{11}+1898x ^{10}-3034x^9 \\
& & + 4137x^8    -3034x^7+1898x^6-2237x^5+2018x^4-538x^3-255x^2+139x+1.
\end{eqnarray*}

On constate d'abord que $Q_3$ est en effet  le polyn\^ome minimal d'une unit\'e de $H_{K}^+.$ Par ailleurs,
six autres invariants co\"incident eux  aussi avec des racines de ce polyn\^ome, au moins pour leurs $n$ premi\`eres d\'ecimales ($10\leq n\leq 200$ selon les cas):
\begin{eqnarray*}
u_1:=u_1(1)e^{\frac{16}{15}i\pi} &\approx&  -140.7834195600\ldots;  \\
u_2:=u_2(1)e^{\frac{19}{15}i\pi} &\approx&  0.589707772431\ldots - 0.271949567981\ldots i;\\
u_3:=u_3(1)e^{\frac{24}{15}i\pi} &\approx& -0.362402166665\ldots; \\
u_4:=u_4(1)e^{\frac5{15}i\pi} &\approx& 0.589707772431\ldots + 0.271949567981\ldots i; \\
u_5:=u_5(1)e^{\frac{19}{15}i\pi} &\approx& -0.007103109180\ldots; \\
u_6:=u_6(1)e^{\frac{3}{15}i\pi} &\approx& 1.398366700490\ldots + 0.644870625513\ldots i; \\
u_7:=u_7(1)e^{\frac{12}{15}i\pi} &\approx&      -2.759365401151\ldots.
\end{eqnarray*}
La pr\'ecision avec laquelle   $\rho_8(\gamma_8)$ est obtenu se r\'ev\`ele
 insuffisante pour identifier $u_8$ avec une des racines de $Q_3.$
 Pour des raisons de sym\'etrie, il doit correspondre  \`a la racine
$$ 1.398366700490\ldots  - 0.644870625513\ldots i.$$

\subsection{Corps de base $\Q (\sqrt{2})$}
L'anneau des entiers de $F'=\Q(\sqrt{2})$ est euclidien pour la norme.
On note $\epsilon=1+\sqrt{2}$ son unit\'e fondamentale de norme $-1, $ et l'on ordonne les plongements de sorte que $\epsilon_1<0$ et $\epsilon_2>0.$ La constante $m_{F'}$ optimale est vraisemblablement $m_{F'}=6$ dans ce cas.

\medskip

\noindent
{\bf (a) Un exemple \`a groupe des classes $C_4.$}

\medskip

L'extension ATR
$K=F'(\sqrt{12\epsilon-11})$ poss\`ede un groupe des classes au sens restreint cyclique d'ordre $4.$
Nous associons \`a chaque classe un invariant dans $\C/m_{F'}\Z$:
\begin{eqnarray*}
\rho_1(\gamma_1) &\approx&-1.333333333333333\ldots - 0.301378336840440\ldots i; \\
\rho_2(\gamma_2) &\approx& -0.274078669810665\ldots; \\
\rho_3(\gamma_3) &\approx&   0.166666666666666\ldots + 0.301378336840440\ldots i; \\
\rho_4(\gamma_4) &\approx&  -3.225921330189334\ldots.
\end{eqnarray*}
Tous sont  obtenus avec une pr\'ecision sup\'erieure
\`a $40$ d\'ecimales.
On en d\'eduit  au moyen de la commande Pari
   \texttt{algdep}
le polyn\^ome  candidat
$$ Q_4(x)=x^8+6x^7-5x^6-4x^5+5x^4-4x^3-5x^2+6x+1.$$
 On v\'erifie a posteriori que ce polyn\^ome d\'efinit bien une unit\'e du corps
de classes de Hilbert au sens restreint de $K.$
Par ailleurs, $4$ des  $8$ racines de $Q_4$  co\"incident sur leurs $40$ premi\`eres
 d\'ecimales  avec les nombres complexes
$$\begin{array}{lll}
u_1:=u_1(1)e^{\frac{10}{6}i\pi}&\approx&  -6.643347233735518\ldots; \\
u_2:=u_2(1)e^{\frac{10}{6}i\pi} &\approx&  -0.931490243381137 \ldots -0.363766307518644\ldots i; \\
u_3:=u_3(1)e^{\frac{4} 6 i\pi} &\approx& -0.150526528994587\ldots;\\
u_4:=u_4(1)e^{\frac{8}{6}i\pi} &\approx&   -0.931490243381137\ldots + 0.363766307518644\ldots i.
\end{array}$$
\noindent

\medskip

\noindent
{\bf (b) Un exemple \`a groupe des classes $C_8.$}

\medskip

On consid\`ere enfin l'extension quadratique ATR
$K=F'(\sqrt{25\epsilon-31}),$ dont  le  groupe des classes au sens restreint est cyclique d'ordre $8$.
\`A chaque classe correspond une matrice $\gamma_k\in \SL_2(\mc O_{F'})$ et  un invariant de $\C/m_{F'}\Z $:
\begin{eqnarray*}
\rho_1(\gamma_1) &\approx&    -3.666666666666\ldots - 2.047636549497\ldots i; \\
\rho_2(\gamma_2) &\approx& 1.855997078695\ldots - 0.315172999961\ldots i;  \\
\rho_3(\gamma_3) &\approx&  -122.347\ldots + 0.625\ldots i; \\
\rho_4(\gamma_4) &\approx&  -87.06066958797\ldots+ 0.315172999961\ldots i; \\
\rho_5(\gamma_5) &\approx&   -18.16666666666\ldots + 2.047636549497\ldots i; \\
\rho_6(\gamma_6) &\approx&  20.060669587971\ldots + 0.315172999961\ldots i; \\
\rho_7(\gamma_7) &\approx&   -13.884816073095\ldots;  \\
\rho_8(\gamma_8) &\approx&  47.894002921304\ldots - 0.3151729999612\ldots i.
\end{eqnarray*}
L'invariant le plus pr\'ecis est $\rho_5(\gamma_5)$, qu'on a pu
calculer avec plus de
 200 d\'ecimales significatives.
La commande Pari   \texttt{algdep($u_5(1)e^{-\frac 8 {16}i\pi},16)$} sugg\`ere  le polyn\^ome r\'eciproque
\begin{eqnarray*}
Q_5(x)&:=&x^{16} + 386792 x^{15} - 5613916 x^{14} +
21963312 x^{13} - 13291318 x^{12} + 32052888 x^{11} \\
& & + 15011472 x^{10} + 16774296 x^{9} + 36336275 x^8 + 16774296 x^7 + 15011472 x^6 \\
 & &  + 32052888 x^5 - 13291318 x^4  + 21963312x^3 - 5613916 x^2+386792 x + 1.
\end{eqnarray*}
Il est ais\'e de v\'erifier que  $Q_5$ d\'efinit effectivement  une unit\'e de $H_{K}^+.$
\`A une racine de l'unit\'e pr\`es, les exponentielles des  nombres complexes pr\'ec\'edents  co\"incident sur leurs premi\`eres d\'ecimales (entre 10 et 200 selon les cas) avec les racines suivantes de $Q_5$:
$$\begin{array}{lll}
u_1:=u_1(1)e^{\frac{2}{6}i\pi} &\approx& -386806.513645927\ldots; \\
u_2:=u_2(1)e^{\frac{2}{6}i\pi} &\approx&  7.17151519909699\ldots + 1.02818667270890\ldots i; \\
u_4:=u_4(1)e^{\frac{1}{6}i\pi} &\approx& 0.136632045175690\ldots + 0.0195890608908274\ldots i; \\
u_5:=u_5(1)e^{\frac{8}{6}i\pi} &\approx& -0.00000258527187\ldots; \\
u_6:=u_6(1)e^{\frac{11}{6}i\pi} &\approx& 0.136632045175690\ldots - 0.0195890608908274\ldots i; \\
u_7:=u_7(1)e^{\frac{7}{6}i\pi} &\approx& -0.317863811003618\ldots - 0.948136381357796\ldots i; \\
u_8:=u_8(1)e^{\frac{1}{6}i\pi} &\approx& 7.17151519909699\ldots  - 1.02818667270890\ldots i.
\end{array}$$
La pr\'ecision obtenue sur les d\'ecimales de $\rho_3(\gamma_3)$ est insuffisante pour identifier $u_3.$ Pour des raisons de sym\'etrie, il doit correspondre  \`a la racine  suivante de $Q_5 :$
$$ -0.317863811003618\ldots + 0.948136381357796\ldots i. $$

\section{P\'eriodes de s\'eries d'Eisenstein.}\label{lastsection}

L'objet de cette partie est  d'\'etablir une formule g\'en\'erale
qui exprime  la valeur sp\'eciale en $s=0$ des fonctions $L$
introduites  pr\'ec\'edemment en terme   de p\'eriodes de s\'eries
d'Eisenstein pour un tore de $\Gamma,$ ce qui  compl\`ete la
d\'emonstration  des th\'eor\`emes \ref{thm:realmain} et
\ref{thm:atrmain}.

\medskip

Des formules similaires  ont d\' ej\`a \'et\'e obtenues dans [Har]
et [Ha] par exemple. On donne ici une pr\'esentation des
r\'esultats expos\'es dans la th\`ese du premier auteur [Ch1,
section 5] sous une forme directement
 utilisable  dans les parties \ref{sec:TR} et \ref{sec:ATR}.

\bigskip

Quelques notations multi-indices standard permettront de rendre les formules plus a\-gr\'e\-ables.
 On associe d'abord \`a un $n$-uplet de nombres  complexes $z=(z_1,\ldots,z_n)\in\C^n$ sa partie imaginaire $y=(\textrm{Im}(z_1),\ldots,\textrm{Im}(z_n)),$ sa trace $Tr(z)=z_1+\ldots +z_n$
et sa norme $N(z)=z_1\cdots z_n.$ Pour un \'element $\mu$ de $F,$
on d\'esigne par $\mu z$ et $z+\mu$ les $n$-uplets
$(\mu_1z_1,\ldots,\mu_nz_n)$ et $(\mu_1+z_1,\ldots,\mu_n+z_n)$
respectivement. On  introduit alors  pour $\Reel(s)>1$ la s\'erie
d'Eisenstein
\begin{eqnarray}
\label{defEzs}
E(z,s)&=&
\sum_{\ \ \ (\mu,\,\nu)\in \mc O_F^2/\mc O_F^\times} \!\!\!\!\!\!\!\!\!\prim
\ \frac{ N(y)^{s} }{|N(\mu
z+\nu)|^{2s}},
\end{eqnarray}
o\`u le groupe d'unit\'es  $\mc O_F^\times$ op\`ere diagonalement
sur $\mathcal O_F^{\,2}.$ Cette s\'erie  d\'efinit  une forme
modulaire de Hilbert non-holomorphe de poids $(0, \ldots,0)$ pour
$\Gamma.$ Le th\'eor\`eme principal de cette partie met en jeu des
p\'eriodes associ\'ees \`a des d\'eriv\'ees partielles de
$E(z,s).$

\bigskip
 Soit $K$ une extension  quadratique de $F.$
 La
signature de $K$ est  de la forme $(2r,\,c)$ avec $r+c=n.$ On
ordonne les $n$ places archim\'ediennes de $F$ de sorte que les
$c$ premi\`eres places  $\upsilon_1,\ldots,\upsilon_c$ se
prolongent chacune  en une place complexe de $K,$ et que pour les
$r$ places suivantes $\upsilon_{c+1},\ldots,\upsilon_n$
  on ait un isomorphisme de $\R$-alg\`ebre
$K\otimes_{F,v_j} \R\simeq \R\oplus \R.$
 On fixe une fois pour toutes de telles identifications, que l'on appelle encore  $\upsilon_j$ par abus de notation.

\medskip

Comme dans la partie \ref{subsec:quad}, on fixe un id\'eal $I$ de
$\mathcal O_F,$ et l'on note  $\mathcal O_I=\mathcal O_F+I\mathcal
O_K$ l'ordre de $K$ de conducteur $I.$ On se donne un  $\mathcal
O_I$-module projectif $M$ de $K,$ et l'on d\'efinit
$\tau:=\omega_2/\omega_1,$ o\`u $(\omega_1, \omega_2)$ est une
 une $\mathcal O_F$-base positive de $M.$  On pose
alors
$$ \left\{
\begin{array}{lll}
\tau_j:= v_j(\tau) & \in \cH_j & \mbox{ pour } j=1,\ldots,c,   \\
(\tau_j,\tau_j'):= v_j(\tau) & \in \R\times \R &  \mbox{ pour } j=c+1,\ldots,n.
\end{array} \right.
$$
Pour chaque $c+1\le j\le n,$ on appelle
 $\Upsilon_j$ la g\'eod\'esique hyperbolique sur $\cH_j$
joignant $\tau_j$ \`a $\tau_j'$, orient\'ee dans le sens
allant de $\tau'_j$ \`a $\tau_j$.

 Le produit
$$ R_\tau = \{\tau_1\}\times \cdots\{\tau_c\}\times  \Upsilon_{c+1} \times\cdots\times\Upsilon_n\subset \mc H ^n$$
est un espace contractile hom\'eomorphe \`a $\R^{r}$.
On
le munit de l'orientation naturelle h\'erit\'ee des
$\Upsilon_j$.
 Le stabilisateur $\Gamma_\tau$ de $\tau$ dans $ \Gamma$ est un groupe ab\'elien
de rang $r$ (modulo la  torsion), qui s'identifie avec le
sous-groupe $V_1$ des unit\'es de $V$ de norme relative  $1$ sur
$F$. Il op\`ere sur $R_\tau$ par homographies, et le quotient
$\Gamma_\tau\backslash R_\tau$ est compact, isomorphe \`a un tore
r\'eel de dimension $r$. Soit $\Delta_\tau$ un domaine fondamental
pour l'action de $\Gamma_\tau$ sur $R_\tau$. On identifie
$\Delta_\tau$ avec son image dans $\cX$, qui est un cycle ferm\'e
de dimension $r$ dans ce quotient.

\medskip

\begin{theorem} \label{theo:ppalannexe}
Pour tout $\cO_I$-module projectif  $M$ dans $K,$ on a:
\begin{equation}
\label{finalEzsperiode}
\int_{\Delta_\tau} \frac{\partial^rE(z,s)}{\partial z_{c+1} \cdots \partial z_{n}} dz_{c+1}\wedge \ldots\wedge  dz_n =  \left(\frac{\Gamma(\frac{s+1} 2)^{2}}{2i\Gamma(s)}\right)^r d_F^{-s}  \, L(M, s).
\end{equation}
\end{theorem}

\begin{proof} On associe \`a tout nombre complexe $s$
 la $r$-forme diff\'erentielle  $\Gamma$-invariante sur
$\mathcal H^n$
$$
\omega_{\textrm{Eis}}^r(s):=\frac{\partial^rE(z,s)}{\partial z_{c+1}\cdots \partial z_{n}} dz_{c+1}\wedge \ldots\wedge  dz_n.$$
Lorsque  $\Reel(s)>1,$ un calcul direct montre que la p\'eriode consid\'er\'ee  prend  la forme
\begin{equation}
\label{premiereperiodeomegar}
\int_{\Delta_\tau} \omega_{\textrm{Eis}}^r(s) = \left(\frac s {2i}\right)^r\int_{\Delta_\tau}\!\!\!\!
\sum_{\ \ \ (\mu,\,\nu)\in \mc O_F^2/\mc O_F^\times}
\!\!\!\!\!\!\!\!\! \prim\!\!
 \left(\prod_{j=1}^{c}
\frac{(\Ima \
\tau_j)^s}{|\mu_j\tau_j+\nu_j|^{2s}}\right)
\left(\bigwedge_{j=c+1}^n \frac{
y_j^{s-1}(\mu_j\bar{z}_j+\nu_j)^2
}{|\mu_jz_j+\nu_j|^{2s+2}}dz_j\right).
\end{equation}

On d\'efinit une action naturelle  de $K^\times$ (et donc du
groupe $V_1$) sur   $(\R^\times_+)^r$ par la formule
\begin{equation*}\alpha\centerdot (t_{c+1},\ldots,t_n) :=  
\left(\left| \frac {\alpha_{c+1}}{\alpha'_{c+1}} \right|
t_{c+1},\ldots, \left| \frac {\alpha_{n}}{\alpha'_{n}} \right|
t_n\right).\end{equation*} Le tore compact r\'eel $T^r:V_1\backslash (\R^\times_+)^r$ est muni d'une mesure de Haar
canonique 
$$d^\times t=\frac{dt_{c+1}}{t_{c+1}}\wedge \ldots\wedge
\frac{dt_{n}}{t_{n}}.$$

On
peut supposer que  $(1,\tau)$ est  une $\mathcal O_F$-base positive de $M$,
quitte \`a changer $M$ en $\alpha M$ avec $\alpha\in K^\times$. 
  Dans ce cas, la g\'eod\'esique  $\Upsilon_j$ est orient\'ee dans le sens trigonom\'etrique selon nos conventions.
 On obtient  une param\'etrisation $t_j\in \R_+^\times$ de cette g\'eod\'esique   en posant
 $t_j=-i\frac{z_j-\tau'_j}{z_j-\tau_j}.$ Elle
  permet d'identifier le quotient 
$\Gamma_\tau\backslash R_\tau$ avec le tore  $T^r$ en 
respectant les orientations.  
On  observe  d'abord que 
$$N(M)=d_F\prod_{j=1}^c\Ima(\tau_j)\prod_{j=c+1}^n (\tau'_j-\tau_j).$$
Le changement de variable  qui correspond 
\`a cette param\'etrisation  transforme   l'identit\'e
 (\ref{premiereperiodeomegar}) en l'expression
\begin{align}\label{identiteintermomegar}
\int_{\Delta_\tau} \omega_{\textrm{Eis}}^r(s) = \left(\frac s
{2}\right)^r  \left(\frac{N(M)}{d_F}\right)^s\!\!
\int_{T^r}\!\!\!\!\!\! \sum_{\ \ \ \beta \in M/\mc O_F^\times}
\!\!\!\!\!\prim \left|N_{K/\Q}(\beta)
\right|^{-s}g_{\beta}(\beta\centerdot t)d^\times t,\end{align}
o\`u l'on a pos\'e $\beta=\mu\tau+\nu,$ \'el\'ement de $K^\times$
qui parcourt les classes non nulles de $M/\mc O_F^\times$ quand le
couple $(\mu,\nu)$ parcourt les classes non nulles de $\mathcal
O_F^2/\mc O_F^\times,$ et  o\`u $g_\beta :  (\R_+^\times)^r
\longrightarrow \C$ d\'esigne  la fonction auxiliaire

$$g_\beta(t)=\prod_{j=c+1}^n \frac{t_j^s(-it_j+\sign(\beta_j\beta'_j))^2}{(t_j^2+1)^{s+1}}.$$
L'\'etape cruciale consiste maintenant \`a utiliser une id\'ee
 due \`a Hecke ([Si] p. 86) :  on observe d'abord  que
l'on obtient un syst\`eme de repr\'esentants des classes non nulles de  $M/\mc O_F^\times$ en consid\'erant la famille $\{\beta\epsilon \}$ lorsque $\beta$ parcourt  les classes non-nulles de $M/\tilde{V}$ et $\epsilon$ parcourt  $V_1/\{\pm 1\}.$

Par cons\'equent, l'identit\'e (\ref{identiteintermomegar}) devient
\begin{align}
\int_{\Delta_\tau} \omega_{\textrm{Eis}}^r(s)
=&\,
 \left(\frac s {2}\right)^r  \left(\frac{N(M)}{d_F}\right)^s\!\!
\sum_{\ \ \ \beta \in M/\tilde{V}}\!\!\! \!\prim\!\!
\left|N_{K/\Q}(\beta) \right|^{-s}\int_{V_1 \backslash(\R_+^\times)^r}\,\sum_{\epsilon\in V_1/\{\pm 1\}}
g_{\beta\epsilon}(\beta\epsilon \centerdot t)d^\times t\\
=&\,
 \left(\frac s {2}\right)^r  \left(\frac{N(M)}{d_F}\right)^s\!\!
\sum_{\ \ \ \beta \in M/\tilde{V}}\!\!\! \! \prim\!\!
\left|N_{K/\Q}(\beta) \right|^{-s}
 \int_{(\R_+^\times)^r}\! g_{\beta}(\beta \centerdot t)d^\times t.
\end{align}
Le changement de variable $u=\beta\centerdot t$ dans la derni\`ere int\'egrale permet de scinder
cette int\'egrale  multiple en  un produit de $r$ int\'egrales de la forme suivante ([Si] formule (107)) :
\begin{equation*} \int_0^{+\infty} \frac{u_j^s(-i
u_j+\sign(\beta_j\beta'_j))^2}{(u_j^2+1)^{s+1}}\frac{du_j} {u_j} -i\sign(\beta_j\beta'_j)\frac{\Gamma\left( \frac{s+1}
2\right)^2}{\Gamma(s+1)}.\end{equation*}
Il s'ensuit  que
\begin{align}
\int_{\Delta_\tau} \omega_{\textrm{Eis}}^r(s)
=&\,
 \left(\frac {s\Gamma(\frac{s+1} 2)^2} {2i\Gamma(s+1)}\right)^r  \left(\frac{N(M)}{d_F}\right)^s\!\!
\sum_{\ \ \ \beta \in M/\tilde{V}}  \!\!\!\!\prim\!\!\!
\left|N_{K/\Q}(\beta) \right|^{-s}\prod_{j=c+1}^n
\sign(\beta_j\beta'_j).
\end{align}
La formule (\ref{finalEzsperiode}) s'en d\'eduit imm\'ediatement au vu de la d\'efinition (\ref{eqn:LMs}) de $L(M,s).$
\end{proof}

\medskip

\begin{corollary}\label{cor:final}
Soit $K$ une extension quadratique de signature $(2r,c)$ du corps  $F,$ avec $r+c=n=[F:\Q].$  Soit $M$ un   $\cO_I$-module projectif  dans $K.$
La fonction $L(M,s)$ poss\`ede alors un z\'ero d'ordre $\geq c$ en $s=0,$ et l'on a les formules :

\begin{itemize}
\item[i)]
si $ r\geq 2 ,$ alors :
\begin{equation}\label{L0r2}
\frac{L^{(c)}(M,0)}{c!}=\frac{(2i)^r }{2^n\pi^r}\int_{\Delta_\tau}  \frac{\partial^r \tilde{h}(z)}{\partial z_{c+1}\ldots \partial z_n} dz_{c+1}\wedge \ldots \wedge dz_n.
\end{equation}
\item[ii)]
si le corps $K$ n'a que deux places r\'eelles ($r=1$), alors
\begin{equation}\label{L0r1}
\frac{L^{(n-1)}(M,0)}{(n-1)!}=\frac{2i}{2^n\pi }\int_{\Delta_\tau}  \left(\frac{\partial \tilde{h}(z)}{\partial  z_n}-\frac {2^{2n-2}R_F} {z_n-\bar{z}_n}\right)dz_n.
\end{equation}
\end{itemize}
\end{corollary}

\begin{proof}
D'apr\`es les r\'esultats de Asai ([As], Th\'eor\`eme 3, ou [Ch2], Th\'eor\`eme 2.1),
la fonction $E(z,s)$ se prolonge sur $\C$ en une fonction m\'eromorphe de la variable $s$ qui satisfait l'\'equation fonctionnelle
\begin{equation}G(2s)E(z,s)=G(2-2s)E(z,1-s)\end{equation} avec $G(s):=d_F^{\frac s2}\pi^{-\frac {ns} 2} \Gamma(\frac s2)^n.$
En outre, elle poss\`ede un unique p\^ole simple en $s=1$,
et les premiers termes de son d\'eveloppement de Laurent au voisinage de ce p\^ole sont fournis par la formule limite de
Kronecker g\'en\'eralis\'ee :

\begin{equation}
E(z,\,s)=\frac{(2\pi)^nR_F}{4d_F}\left(\frac 1 {s-1}+ \gamma_F-
\log\left(\prod_{j=1}^n
y_j\right)+h(z)\right)+O(s-1),\end{equation} o\`u $\gamma_F$ est une constante  qui ne d\'epend  que de $F,$ et
o\`u les fonctions $h$ et $\tilde{h}=4^{n-1}R_Fh$ ont   \'et\'e introduites en (\ref{devhAsai}) et
(\ref{devhAsaiprime}).
Les deux \'egalit\'es pr\'ec\'edentes permettent d'obtenir le d\'eveloppement de Taylor  de $E(z,s)$ au voisinage de $s=0$ :
\begin{equation}E(z,s)=-2^{n-2}R_F s^{n-1}-2^{n-2}R_F s^n\Big(\log N(y)+ \gamma'_F-h(z)\Big)+O(s^{n+1}),\end{equation}
où $\gamma'_F$ ne dépend que de $F.$
On trouve  par cons\'equent  pour $r=n-c\geq 2$ :
\begin{equation}\label{eq1:LMSr2}\frac{\partial^r E(z,s)}{\partial z_{c+1}\ldots \partial z_n}= \frac{s^n}{2^n} \frac{\partial^r \tilde{h}(z)}{\partial z_{c+1}\ldots \partial z_n} +O(s^{n+1}),\end{equation}
et pour $r=1$ :
\begin{equation}\label{eq2:LMSr1}\frac{\partial E(z,s)}{\partial z_n}=\frac{s^n}{2^n} \left(\frac{\partial \tilde{h}(z)}{\partial z_{n}}-\frac {2^{2n-2}R_F} {z_n-\bar{z}_n}\right)+O(s^{n+1}).\end{equation}

On conclut alors  du th\'eor\`eme \ref{theo:ppalannexe} que $L(M,s)$ se prolonge en une fonction
 m\'eromorphe  sur $\C,$ dont le   d\'eveloppement de Taylor au voisinage de  $s=0$ se d\'eduit des identit\'es (\ref{eq1:LMSr2}) et (\ref{eq2:LMSr1}) ci-dessus :

 $$ L(M, s)=\frac{(2i)^r s^{n-r}}{2^n \pi^r }
 \int_{\Delta_\tau}  \frac{\partial^r \tilde{h}(z)}{\partial z_{c+1}\ldots \partial z_n}dz_{c+1}\wedge \ldots \wedge dz_n+O(s^{n+1-r})$$
 pour $r\geq 2,$ tandis que pour  $r=1$ :

 $$L(M, s)=\frac{2is^{n-1}}{ 2^{n}\pi }\int_{\Delta_\tau} \left(\frac{\partial \tilde{h}(z)}{\partial z_{n}}-\frac {2^{2n-2}R_F} {z_n-\bar{z}_n}\right)+O(s^{n}).$$

  Les formules (\ref{L0r2}) et (\ref{L0r1}) souhait\'ees en r\'esultent  imm\'ediatement.\end{proof}


\bigskip




\begin{thebibliography}{XXX}


\bibitem[As]{Asai} T.~Asai. {\em On a certain function analogous to
$\log|\eta(z)|.$} Nagoya Math. J.  \textbf{40} (1970),
193-211.



\bibitem[Co]{Cooke}
G.~Cooke. {\em A weakening  of the Euclidean property for integral domains and applications to algebraic number theory, Part I.} J. Reine Angew. Math. \textbf{282} (1976), 133-156.



\bibitem[Ch1]{charollois}
P.~Charollois. {Sommes de Dedekind et p\'eriodes de formes
modulaires de Hilbert.} Th\`ese de doctorat, Universit\'e
Bordeaux I, 2004.



\bibitem[Ch2]{charolloispapier}
P.~Charollois. {\em Sommes de Dedekind associées à un corps de
nombres totalement réel}. J. Reine Angew. Math. \textbf{610}  (2007),  125-147.


\bibitem[Dar1]{darmon-hpxh}
H.~Darmon. {\em Integration on $\cH\sb p\times\cH$ and
arithmetic applications}.  Ann. of Math. (2)  {\bf 154}  (2001),  no. 3, 589--639.

\bibitem[Dar2]{darmon1}
H.~Darmon. {Rational points on modular elliptic curves}. CBMS
Regional Conference Series in Mathematics  \textbf{101}, AMS-NSF:
 2004.


\bibitem[Das]{dasgupta_senior}
S.~Dasgupta. {\em Stark's conjectures}. Harvard senior thesis
(1999).

\bibitem[DD]{darmon-dasgupta}
H.~Darmon et S.~Dasgupta. {\em Elliptic units for real quadratic
fields}. Annals of mathematics   \textbf{163} (2006),  301-345.


\bibitem[DL]{darmon-logan}
H.~Darmon et A.~Logan. {\em Periods of Hilbert modular forms and
rational points on elliptic curves}. International Mathematics
Research Notices  \textbf{40} (2003), 2153-2180.


\bibitem[dSG]{deshalit-goren}
E.~de Shalit et E.Z.~Goren. {\em On special values of theta functions of genus two}.
Ann. Inst. Fourier (Grenoble)  \textbf{47}  (1997),  no. 3, 775--799.



\bibitem[DTvW]{dtv}
D.S.~Dummit, B.A.~Tangedal et P.B.~van Wamelen. {\em Stark's
conjecture over complex cubic number fields}. Math. Comp.
\textbf{73} (2004), no. 247, 1525--1546.


\bibitem[Fr]{freitag}
E.~Freitag. Hilbert Modular Forms. Springer Verlag: 1990.

\bibitem[Gr]{greenberg}
M.~Greenberg. {\em Stark-Heegner points and the cohomology of
quaternionic Shimura varieties}. Soumis.

\bibitem[GL]{goren-lauter}
E.Z.~Goren et  K.E.~Lauter.
{\em Class invariants for quartic CM fields}.
Ann. Inst. Fourier (Grenoble)  \textbf{57}  (2007),  no. 2, 457--480.


\bibitem[Ha]{Ha}
 Y.~Hara.  {\em On calculation of $L_K(1,\chi)$ for some Hecke
characters.} J. Math. Kyoto Univ. \textbf{33} (1993), 865-898.


\bibitem[Har]{Har}
 S.~Haran.  {\em  $p$-adic $L$-functions for modular forms.}  Compositio Math. \textbf{62} (1987), no. 1 p. 31-46.



\bibitem[Hard]{harder}
G.~Harder. {\em On the cohomology of $SL(2,O)$}.
 Lie groups and their representations
(Proc. Summer School on Group Representations of the Bolyai J\'anos Math. Soc.,
Budapest, 1971),  pp. 139--150. Halsted, New York, 1975.


\bibitem [He]{He} E.~Hecke. {\em Analytische Funktionen und algebraische
Zahlen, II.} \textbf{20}, Math. Werke, 2nd ed. Vandenhoeck,
Ruprecht, Göttingen: 1970. 381-404.



\bibitem [K-M]{K-M} R.~Kirby et P.~Melvin.  {\em Dedekind sums, $\mu$-invariant
and the signature cocycle.} Math. Ann. \textbf{299}  (1994), no. 2 ,
231-267.



\bibitem [M]{M} H.~Maennel. {\em Denominators of Eisenstein classes on Hilbert modular varieties and the $p$-adic class number formula.} Bonner Math. Schriften \textbf{247}  (1993).


\bibitem[RS]{ren-sczech}
T.~Ren. et R.~Sczech. {\em A refinement of Stark's conjecture over
complex cubic number fields.} Soumis.


\bibitem [Si]{Si}
  C.L.~Siegel. Advanced Analytic Number Theory. Tata
Institute of Fundamental Research, Bombay: 1980.



\bibitem [St]{St} H.M.~Stark. {\em Values of L-functions at $s=1.$} $I,$
$II$, $III$ et $IV,$ Advances in Math. \textbf{7} (1971), 301-343;
\textbf{17} (1975), 60-92; \textbf{22} (1976), 64-84; \textbf{35}
(1980), 197-235.


\bibitem [Ta]{Ta}  J.T.~Tate. Les conjectures de Stark sur les fonctions
$L$ d'Artin en $s=0.$ Birkhaüser, Boston: 1984. Progress in Math.
\textbf{47}.

\bibitem[Tr]{trifkovic}
M.~Trifkovic.
{\em Stark-Heegner points on elliptic
 curves defined over imaginary quadratic fields}.
Duke Math. J.  \textbf{135}  (2006),  no. 3, 415--453.

\bibitem [VdG]{VdG}
G.~Van der Geer. Hilbert modular surfaces. Springer-Verlag, Berlin
Heidelberg New-York: 1988. Ergebnisse der Math. 3. Folge, vol.
\textbf{16}.


\end{thebibliography}
\end{document}